\documentclass[10pt]{amsart}
\usepackage[utf8]{inputenc}
\usepackage[T2A]{fontenc}
\usepackage {amsfonts, amssymb, amscd,amsthm, amsmath, eucal}

\usepackage{amsbsy}
\usepackage[all]{xy}

\usepackage{color}

\newtheorem{theorem}{Theorem}[section]
\newtheorem{definition}[theorem]{Definition}
\newtheorem{notation}[theorem]{Notation}
\newtheorem{notations}[theorem]{Notations}
\newtheorem{lemma}[theorem]{Lemma}
\newtheorem{corollary}[theorem]{Corollary}
\newtheorem{proposition}[theorem]{Proposition}
\newtheorem{note}[theorem]{Note}
\newtheorem{remark}[theorem]{Remark}
\newtheorem{conditions}[theorem]{Conditions}
\newtheorem{example}[theorem]{Example}
\newtheorem*{plan}{Plan of proof}

\DeclareMathOperator{\Hom}{Hom}
\DeclareMathOperator{\colim}{colim}

\renewcommand{\leq}{\leqslant}
\renewcommand{\geq}{\geqslant}
\renewcommand{\phi}{\varphi}

\renewcommand{\kappa}{\varkappa}

\newcommand{\bb}{\mathbb}
\newcommand{\A}{\mathbb{A}}

\newcommand{\Gm}{{\mathbb{G}_m}}
\newcommand{\ZF}{\operatorname{\mathbb{Z}F}}

\newcommand{\Fr}{\operatorname{Fr}}
\newcommand{\id}{\operatorname{id}}

\newcommand{\cc}{\mathcal}
\newcommand{\F}{\operatorname{F}}

\newcommand{\pair}{\mathbb}
\newcommand{\spc}{}

\renewcommand{\l}{\left}
\renewcommand{\r}{\right}

\renewcommand{\phi}{\varphi}
\renewcommand{\leq}{\leqslant}
\renewcommand{\geq}{\geqslant}
\renewcommand{\epsilon}{\varepsilon}

\renewcommand{\kappa}{\varkappa}

\newcommand{\op}{{\textrm{\rm op}}}

\usepackage{tikz-cd}

\begin{document}

\sloppy

\title{A motivic Segal theorem for open pairs of smooth schemes over an infinite perfect field}
\author{Aleksei Tsybyshev}
\thanks{The author thanks professor I. Panin for the statment of the problem and valuable discussions, and also A. Mingazov, for his notes~\cite{Min} on this problem, which contain a proof of the main theorem in the case of a smooth closed subscheme, and which inspired some of the arguments in the proof of the main theorem.}

\begin{abstract}
%В.Воеводский в своих неопубликованных заметках заложил основы машинерии распетливания мотивных пространств, введя понятия оснащенных соответствий,
%оснащенных пучков и доказав их базовые свойства. Одна из целей этой машинерии -- дать новую конструкцию стабильной мотивной категории $SH(k)$
%($k$ -- совершенное поле), конструкцию лучше адаптированную для вычислений.
V. Voevodsyky laid the groundwork of delooping motivic spaces in order to provide a new, more computation-friendly, construction of the stable motivic category $SH(k)$,
G. Garkusha and I. Panin made that project a reality, while collaborating with A. Ananievsky, A. Neshitov and A. Druzhinin.
In particular, G. Garkusha and I. Panin proved that for an infinite perfect field $k$ and any $k$-smooth scheme $X$
the canonical morphism of motivic spaces
$C_*Fr(X)\to \Omega^{\infty}_{\bb{P}^1} \Sigma^{\infty}_{\bb{P}^1} (X_+)$
%$\Omega^{\infty}_{\bb{A}^1} \Sigma^{\infty}_{\bb{A}^1} (X_+)\to \Omega_{\bb{A}^1}(C_*Fr(X \wedge \pair{T})_f)$
is Nisnevich-locally a group-completion.

In the present work, a generalisation of that theorem to the case of smooth open pairs
$(X,U),$ where $X$ is a $k$-smooth scheme, $U$ is its open subscheme intersecting each component of $X$ in a nonempty subscheme. We claim that in this case the motivic space $C_*Fr((X,U))$ is Nisnevich-locally connected, and the motivic space morphism
$C_*Fr((X,U))\to \Omega^{\infty}_{\bb{P}^1} \Sigma^{\infty}_{\bb{P}^1} (X/U)$ is Nisnevich-locally a weak equivalence.
Moreover, we show that if the codimension of $S=X-U$ in each component of $X$ is greater than $r \geq 0,$
the simplicial sheaf $C_*Fr((X,U))$ is locally $r$-connected.

\end{abstract}

\maketitle{}

\tableofcontents

\section{Introduction}
Recall that in~\cite[Section~2]{VoNotes} Voevodsky defined for an arbitrary field $k$ a category of framed correspondences
$Fr_*(k)$, whose objects are the same as the objects of $Sm/k$, and the set of morphisms
$Fr_*(X,Y)=\sqcup_{n\geq 0} Fr_n(X,Y)$ is defined as the set of some geometric data.
In the same notes, framed presheaves are defined as presheaves on the category $Fr_*(k),$
and the notion of a Nisnevich sheaf with framed transfers is introduced. Those are the presheaves with framed transfers, which become Nisnevich sheaves in restriction to $Sm/k$. In the same notes, many basic properties both of the category $Fr_*(k)$ and presheaves or Nisnevich sheaves with transfers are proved.

Starting from these fundamental notes by Voevodsky, a theory of framed motives was constructed in \cite{GP}.
In particularm it was proved in~\cite{GP} that for any infinite perfect field $k$ and any $k$-smooth scheme $X$
the canonical morphism of spaces
$C_*Fr(X)\to \Omega^{\infty}_{(\bb{P}^1,\infty)} \Sigma^{\infty}_{(\bb{P}^1,\infty)} (X_+)$
is Nisnevich-locally a group-completion. If the $k$-smooth scheme $X$ is substituted with a simplicial $k$-smooth scheme $Y$ such that the motivic space $C_*Fr(Y)$ is locally connected, the canonical morphism of motivic spaces
\[C_*Fr(Y)\to \Omega^{\infty}_{(\bb{P}^1,\infty)} \Sigma^{\infty}_{(\bb{P}^1,\infty)} (Y_+)\]
is Nisnevich-locally a weak equivalence.

\section{Statement of the main theorem and its reduction to the other statements}

First we lay out some notatons, conventions, and recall the necessary definitions. These definitions, as well as the definitions of other framed objects, are taken from~\cite[Section~2]{VoNotes} and~\cite{GP}. 

Throughout the paper, let $k$ be an infinite perfect field, and $char(k) \neq 2.$

\begin{note}
Owing to~\cite[Theorem 1.1]{DP}, one can get rid of the characteristic restriction on $k$.
\end{note}
For the reader familiar with the general notions and notation of the theory of framed correspondences, who wants to go straight to the statements of the results, ignoring the preliminaries, the list of the notations that are not commonly used in the same way as in this paper is given below.

\begin{notations}\label{n:general}
The object (Usually $\pair{B}$ or $\pair{T}$), written in a blackboard bold font, denotes an open pair of smooth varieties. In case of $\pair{B}$ it denotes some $(X,U) \in SmOp(Fr_0(k)),$ where $U=X-S,$ and in case of $\pair{T}$ this is a specific object --- the pair $(\A^1,\Gm).$ (see Definition~\ref{definition_pair}, Notation~\ref{notation_T_pair})

The object (Usually $\spc{B}$ or $\spc{T}$), written with the same letter in normal font, denotes a pointed Nisnevich sheaf one gets from the corresponding blackboard bold pair by applying the functor $spc.$ In case of $\spc{B}$ this is $X/U,$ and in case of $\spc{T}$ this is the usual sheaf $\A^1/\Gm.$ (see Notation~\ref{notation_spc}, Note~\ref{note_T})

The functor $Fr_m(-,\pair{B})$ is defined geometrically (see Definition~\ref{d:Fr-pair}), but, owing to the Voevodsky Lemma, it can be understood as $\cc{F}r_m(-,\spc{B}).$ (see Note~\ref{note_lemma_Voev}

Any smooth variety $X$ is also understood as $(X,\emptyset).$ With such an identification, we get $spc(X)=Hom_{Sm}(-,X_+).$ (See Note~\ref{note_X_+})

Most of the time in this paper, the targets of framed correspondences are not motivic spaces, but instead simplicial pairs, objects of $\Delta^{op}SmOp(Fr_0(k)).$
\end{notations}

There are two important conventions. Let $sShv_{\bullet}(Sm/k)$ be the category of simplicial Nisnevich sheaves on the category of $k$-smooth schemes $Sm/k$
(it is also often called the category of motivic spaces).
According to~\cite[2.7]{Jar1}, the category $sShv_{\bullet}(Sm/k)$ of simplicial Nisnevich sheaves is endowed with an injective local model structure, in which the cofibrations are the monomorphisms,
and the weak equivalences are the local weak equivalences. Choose a functorial fibrant replacement
$\mathcal X\mapsto \mathcal X_f$
with respect to the {\it injective model structure}
on the category $sShv_{\bullet}(Sm/k)$ and use it throughout the paper.

By the motivic model structure on the category of simplicial pointed Nisnevich sheaves we understand the model category of Morel and Voevodsky from~\cite{MV}.

Consider a pointed Nisnevich sheaf (not to be confused with the notation for open pairs below) $(\bb{P}^1,\infty).$
V. Voevodsky in~\cite[Раздел 5]{VoCongr} defined motivic spectra, in particular, $(\bb{P}^1,\infty)-$spectra.
In such papers as~\cite{GP} and~\cite{GNP}, the same objects are called $\bb{P}^1-$spectra,
and the pointed sheaf $(\bb{P}^1,\infty)$ is denoted $\bb{P}^{\wedge 1}.$ In the name of internal consistency, and following Voevodsky's original definition, in this paper we use the notation $(\bb{P}^1,\infty).$
The symbol $T$ denotes, customarily, the pointed Nisnevich sheaf $\pair{A}^1/(\pair{A}^1-\{0\})$.
\begin{definition} \label{definition_pair}
$SmOp(Fr_0(k))$ is the category of pairs $\pair{B}=(X,U),$ where $X$ is a $k$-smooth scheme, and $U$ its open subscheme.
The morphisms from $(X,U)$ to $(X',U')$ are those level $0$ framed correspondences from $X$ to $X'$ which induce correspondences from $U$ to $U'.$

The formula $(X,U) \wedge (X',U') = (X \times X',(X \times U' \cup U \times X'))$ defines a symmetric monoidal category structure on $SmOp(Fr_0(k)).$
\end{definition}

We give a motivation for refining (changing) one of the basic notations from the paper \cite{GP}.
In \cite[Definition 2.5, 2.8]{GP} for each pair $(X,X-S)\in SmOp(Fr_0(k))$, each $U\in Sm/k$ and each $m\geq 0$ the pointed sets $Fr_m(U,X/(X-S)$  
%$Fr_+(U,X/(X-S)):=\bigvee_{m\geq 0}Fr_m(U,X/(X-S))$ 
and $Fr(U,X/(X-S))$ are defined. $Fr(-,X/(X-S))$ is also checked, and confirmed, to be a framed presheaf, and even Nisnevich sheaf.

In~\cite[Definition 5.2.(2)]{GP}, taking the framed motive $M_{fr}(X/(X-S))$ of the pair $(X,X-S)\in SmOp(Fr_0(k))$ is defined as a covariant functor from the category $SmOp(Fr_0(k))$  with values in the category of $S^1$-spectra in the category $sShv_{\bullet}(Sm/k)$. This functor on the category $SmOp(Fr_0(k))$
is constructed using simplicial sheaves of the type $Fr(-,X/(X-S))$, however the symbol $X/(X-S)$ is often used to denote the factor-sheaf, not the pair $(X,X-S)$. To avoid mixing together the two concepts and their notations, we decided to introduce for a pair
$(X,X-S)\in SmOp(Fr_0(k))$
the following notation.
\begin{notation}\label{Key_Notation}
Introduce the new notation: \\
we write $Fr_m(U,(X,X-S))$ instead of $Fr_m(U,X/(X-S))$, \\
%писать $Fr_m(-,(X,X-S))$ вместо $Fr_m(-,X/(X-S)$ и писать
$Fr(U,(X,X-S))$ instead of $Fr(U,X/(X-S))$,\\
where $Fr_m(U,X/(X-S))$ 
%$Fr_m(-,X/(X-S)$ 
and $Fr(U,X/(X-S)$ are defined in \cite[Definition 2.5]{GP}, 
%\cite[Construction 2.6]{GP} 
and \cite[Definition 2.8]{GP}
respectively.
\end{notation}
The new notations allow us to define a covariant functor
\[Fr: SmOp(Fr_0(k))\to Shv_{\bullet}(Fr_*(k))\]
with the rule $Fr(X,X-S)=Fr(-,(X,X-S))$. Note that a priori we might have another covariant functor
\[\mathcal Fr \circ spc: SmOp(Fr_0(k))\to Shv_{\bullet}(Fr_*(k)),\]
taking the pair $(X,X-S)$ to the framed sheaf $\mathcal Fr(-, X/(X-S))$.
Here $X/(X-S)$ is the pointed Nisnevish quotient sheaf,  and for an arbitrary pointed Nisnevich sheaf $\mathcal F$ $\mathcal Fr(-,\mathcal F)$ denotes the framed Nisnevich sheaf defined in~\cite[Definition 3.10]{GP}).

\begin{note} \label{note_lemma_Voev}
We stress that according to the corollary~\cite[Isomorphisms (6)]{GP} from the Voevodsky Lemma, these two functors
\[Fr,~(\mathcal Fr \circ spc): SmOp(Fr_0(k))\rightrightarrows Shv_{\bullet}(Fr_*(k))\]
are canonically isomorphic. In particular, they have the same properties.
Sometimes one of them is more useful to work with, sometimes the other, but both functors are defined on pairs, specifically on the category $SmOp(Fr_0(k))$.
\end{note}

We recall here~\cite[Definitions 2.5, 2.8]{GP}, using the changed notations of~\ref{Key_Notation}.

\begin{definition}\label{d:Fr-pair}\cite[Definitions 2.5, 2.8]{GP}

{\rm
(I) Let $Y$ be a scheme, let $S\subset Y$ be a closed subset, and let $U$ be a scheme. An {\it explicit level $m\geq 0$ framed correspondence}
 from $U$ to $(Y,Y-S)$} consists of tuples:
   $$(Z,W,\phi_1,\ldots,\phi_{m};g:W\to Y),$$
where $Z$ is a closed subset in $U\times\bb A^m$, finite over $U$ (understood as the underlying reduced scheme),
$W$ is an \'etale neighbourhood of $Z$ in $U\times\bb A^m$,
$\phi_1,\ldots,\phi_{m}$ are regular functions on $W$, $g$ is a regular map such that $Z=Z(\phi_1,\ldots,\phi_{m})\cap g^{-1}(S)$ (set-theoretically).
The scheme $Z$ is called the {\it support \/} of the explicit framed correspondence. we also use fours $c = (Z,W,\phi;g)$
to denote explicit framed correspondences.

(II) Two explicit  level $m$ framed correspondences $(Z,W,\phi;g)$ and
$(Z',W',\phi';g')$ are {\it equivalent\/} if 
$Z=Z'$ and there exists an \'etale neighbourhood $W''$ of $Z$ in
$W\times_{\bb A^m_U}W',$ such that  $\phi\circ pr$ coincides with 
$\phi'\circ pr'$ and the morphism $g\circ pr$ coincides with $g'\circ
pr'$ on $W''$.

(III) A {\it level $m$ framed correspondence from $U$ to
$(Y,Y-S)$\/} is an equivalence class of explicit level $m$ framed correspondences from $U$ to $(Y,Y-S)$. Denote by
$Fr_m(U,(Y,Y-S))$ the set of level $m$ framed correspondences $m$ from $U$ to $(Y,Y-S)$. We view it as a pointed set with distinguished point g $0_{(Y,Y-S),m}$ given by the explicit framed correspondence $(Z,W,\phi;g)$ with $W=\emptyset$.

(IV) If $S=Y,$ the pointed set $Fr_m(U,(Y,Y-S))$
coincides with the set
$Fr_m(U,Y)$ of level $m$ framed correspondences from $U$ to $Y$.

(V)
Denote by
%\begin{equation}\label{eq:Sigma}
$\sigma_{(Y,(Y-S))}: Fr_m(U,(Y,(Y-S)))\to Fr_{m+1}(U,(Y,(Y-S)))$
%\end{equation}
the map which takes $\Phi = (Z,W,\phi;g)$ to $(Z\times \{0\},W\times \A^1,\phi\circ pr_W,pr_{\A^1};g)$.
%\smallskip
Set $\sigma:=\sigma_{(Y,(Y-S))}$ and, following~\cite[Definition 2.8]{GP}, call the set 
%\begin{multline*}
\[Fr(U,(Y,(Y-S))):= 
\colim[Fr_0(U,(Y,(Y-S)))\xrightarrow{\sigma} Fr_1(U,(Y,(Y-S))) \xrightarrow{\sigma} \dots ] \]
%\xrightarrow{\sigma_{(Y,(Y-S))}} Fr_2(U,(Y,(Y-S))) \cdots)
%\end{multline*}
the set of \textit{stable framed correspondences} from $U$ to $(Y,(Y-S))$.

\end{definition}

\begin{definition} \label{ZF}
Having given this definition, one can readily extend the~\cite[Definitions 8.3, 8.4, 8.5, and 8.7]{GP} to this situation, yielding linear framed correspondences $\bb ZF_n(B,(X,U))$, with an external product on them, the suspension map $\Sigma:\bb ZF_n(B,(X,U)) \to \bb ZF_{n+1}(B,(X,U)),$ the stable linear framed correspondences $\bb ZF(B,(X,U)),$ and the linear framed movive $LM_{fr}((X,U)).$

 Note that $\bb ZF_n(B,(X,U))$ is a free abelian group with basis consisting of framed correspondences with connected support. We call that basis set $F_n(B,(X,U)).$ That object is 'bad', in the sense that it does not behave well under base change of $B$ and composition, but it is still useful for us.
\end{definition}

\begin{definition}
(\cite[Definitions 2.4, 2.5, 2.11, Remark 2.12]{GPHoInv}) 
The exterior composition of framed correspondences (or linear framed correspondences) between objects of $Sm/k$ defines categories:

$Fr_*$ is the category with objects $Ob(Sm/k)$ and morphisms 
\[Fr_*(X,Y)=\coprod \limits_{m \geq 0} Fr_m(X,Y).\]

$Fr_+$ is the category with objects $Ob(Sm/k)$ and morphisms 
\[Fr_+(X,Y)=\bigvee \limits_{m \geq 0} Fr_m(X,Y).\]

$\bb ZF_*$ is the category with objects $Ob(Sm/k)$ and morphisms 
\[\bb ZF_*(X,Y)=\bigoplus \limits_{m \geq 0} \bb ZF_m(X,Y).\]

A \textit{framed presheaf} is a presheaf on the category $Fr_+.$ A framed presheaf $\cc F$ of abelian groups is called \textit{radditive} if $\cc F (X_1 \coprod X_2)=\cc F (X_1) \oplus \cc F(X_2).$ 

This is the same as the presheaf $\cc{F}$ coming from a presheaf on the category $\bb ZF_*.$  
\end{definition}

We also need the following notion, due to Voevodsky, given in~\cite{GPHoInv}
\begin{definition}(=\cite[Definition 2.7]{GPHoInv})
A $Fr_+$-presheaf $\cc F$ of Abelian groups is
stable if for any $k$-smooth variety the pull-back map
$\sigma^*_X: \cc F(X) \to \cc F(X)$
equals the identity map, where
$\sigma_X=(X\times 0, X\times \bb A^1, t; pr_X) \in Fr_1(X,X)$.
In turn, $\cc F$ is quasi-stable if for any $k$-smooth variety the pull-back map
$\sigma^*_X: \cc F(X) \to \cc F(X)$
is an isomorphism.
\end{definition}
%_{(Y,(Y-S))}

%Благодаря лемме Воеводского в варианте~\cite[Corollary 3.6]{GP}, $Fr_m(S,(X,U)) \simeq Fr_m(S, X/U).$
%Тем самым, в частности, все конструкции в терминах фрейм-соответствий в (симплициальные) пары можно понимать
%в терминах фрейм-соответствий в соответствующие (симплициальные) пунктированные пучки Нисневича.

\begin{definition}
Let $Shv_{\bullet}(Sm/k)$ be the category of pointed Nisnevich sheaves on $Sm/k$.
There is a functor $spc: SmOp(Fr_0(k))\to Shv_{\bullet}(Sm/k)$ taking the pair $(X,U)$ to the quotient  pointed Nisnevich sheaf $X/U$ with the distinguished point $U/U$. If $U=\emptyset$, then by definition $X/U=X_+$.

The functor $spc: SmOp(Fr_0(k))\to Shv_{\bullet}(Sm/k)$ induces a functor
$spc: \Delta^{op}SmOp(Fr_0(k))\to sShv_{\bullet}(Sm/k)$
taking the object $[n]\mapsto (Y_n,U_n)$ to the simplicial pointed Nisnevich sheaf
$[n]\mapsto (Y_n/U_n)$.
\end{definition}

\begin{notation}\label{notation_spc}
For a pair denoted with a blackboard bold font, such as $\pair{B}=(X,U),$ we use the same letter in normal font, such as $B$, to denote the quotient pointed sheaf $X/U$, with distinguished point $U/U.$
\end{notation}

\begin{note} \label{note_X_+}
Each smooth $X \in Fr_0(k)$ has a canonical corresponding pair $(X,\emptyset) \in SmOp(Fr_0(k)).$
By abuse of notation, that pair will also be denoted by $X.$

This map gives an inclusion of monoidal categories: $(X \times X', \emptyset)=(X,\emptyset) \wedge (X',\emptyset).$

With such notation, $spc(X)$ is the sheaf $X_+$ in the usual sense.
\end{note}

\begin{notation} \label{notation_T_pair}
The pair $(\bb{A}^1,\bb{G}_m)$ is especially important. We denote it by $\pair{T}$
\end{notation}

\begin{note}\label{note_T}
$spc(\pair{T})$ is the sheaf $\spc{T}=\bb{A}^1/\bb{G}_m$ introduced in Morel and Voevodsky's paper~\cite{MV}.
\end{note}

For a sheaf $U\mapsto {\cc F}(U)$ on $Fr_*(k)$ there is a simplicial sheaf $C_* (\cc F)$ equal to $U\mapsto {\cc F}(\Delta^{\bullet}\times U)$, where $\Delta^{\bullet}$ is the standard cosimplicial scheme $\Delta^n \mapsto \A^n$.
The following key definition is given in~\cite[Section 4]{GP}.
\begin{definition}
%$M_{fr}(\pair{B})$ --- это $S^1$-спектр
%\[(C_* Fr (-,\pair{B}), C_* Fr (-,\pair{B} \otimes S^1), C_* Fr (-,\pair{B} \otimes S^2)\cdots )\]
For each pair $\pair{B}=(X,U) \in SmOp(Fr_0(k))$
we define the $(\bb{P}^1,\infty)$-spectrum $M_{(\bb{P}^1,\infty)}(\pair{B})$ as follows.
\[M_{(\bb{P}^1,\infty)}(\pair{B})=(C_* Fr (-,\pair{B}), C_* Fr (-,\pair{B} \wedge \pair{T}), C_* Fr (-,\pair{B} \wedge \pair{T}^{\wedge 2})\cdots )\]
where the structural morphisms are $C_*(\sigma_n)$,
and where
$\sigma_n: Fr(-,\pair{B}\wedge \pair{T}^n)\to \underline{\Hom}({(\bb{P}^1,\infty)},Fr(-,\pair{B}\wedge \pair{T}^{n+1}))$
are defined in \cite[Section 4]{GP}.
\end{definition}

There is a canonical morphism of $(\bb{P}^1,\infty)$-spectra:
$$\kappa: \Sigma^{\infty}_{(\bb{P}^{1},\infty)}B\to M_{(\bb{P}^1,\infty)}(\pair{B}),$$
given by the identity morphism $\id_{\pair{B}}\in Fr_0(\pair{B},\pair{B})$.
%(recall that a morphism
%from the suspension spectrum of the sheaf $X/U\in Sm/k$ to any other
%spectrum is fully determined by a section of the zeroth motivic
%space of the spectrum at $X$).
%Согласно~\cite[2.7]{Jar1} категория симплициальных пучков Нисневича на
%$Sm/k$ снабжена инъективной локальной модельной структурой, в которой корасслоения -- это мономорфизмы,
%а слабые эквивалентности -- это локальные слабые эквивалентности.
Take the fibrant replacement
$C_*Fr(-,\pair{B}\wedge \pair{T}^n)\to C_*Fr(-,\pair{B}\wedge \pair{T}^n)_f$ of each of the motivic spaces with respect to the injective local model structure.
We then get the $(\bb{P}^1,\infty)$-spectrum
   $$M_{(\bb{P}^1,\infty)}(\pair{B})_f=(C_*Fr(-,\pair{B})_f,C_*Fr(-,\pair{B}\wedge  \pair{T})_f,C_*Fr(-,\pair{B}\wedge  \pair{T}^2)_f, ... ).$$
Note that $M_{(\bb{P}^1,\infty)}(\pair{B})_f$ is the fibrant replacement of the
$(\bb{P}^1,\infty)$-spectrum $M_{(\bb{P}^1,\infty)}(\pair{B})$ with respect to the levelwise injective local model structure on the category of $(\bb{P}^1,\infty)$-spectra.
Let
   $$\kappa_f:\Sigma^{\infty}_{(\bb{P}^1,\infty)}B\to M_{(\bb{P}^1,\infty)}(\pair{B})\to M_{(\bb{P}^1,\infty)}(\pair{B})_f$$
be the composition of the morphisms.

To state the third item of the Theorem below, we recall another definition.
For a $(\bb P,\infty)$-spectrum $E$ let $\mathcal E$ be the $\Omega$-spectrum motivically stably equivalent to $E$.
By $\Omega^{\infty}_{(\bb P^1,\infty)}(E)$ we mean the motivic space
$\cc E_0$ (the zeroth space of the $(\bb P,\infty)$-spectrum $\mathcal E$).
If $E=\Sigma_{(\bb P^1,\infty)}^\infty\cc X$ is the $(\bb P^1,\infty)$-suspension spectrum of the motivic space $\cc X$, we take $\Omega^{\infty}_{(\bb P^1,\infty)}\Sigma^{\infty}_{(\bb P^1,\infty)}(\cc X)$,
to mean
$\Omega^{\infty}_{(\bb P^1,\infty)}(E)$.\label{ominfty}

\begin{theorem} \label{Segal_Thm_II} (cf.~\cite[Theorem 4.1]{GP})
Let $X$ be a $k$-smooth scheme, $S$ be a closed subscheme in $X,$ not containing whole connected components of $X$.
Consider the pair $\pair{B}=(X,X-S) \in SmOp(Fr_0(k))$ and the corresponding motivic space $\spc{B}=X/(X-S)\in Shv_{\bullet}(Sm/k)$
%SmOp(Fr_0(k))$
The following is true:
\begin{enumerate}

\item The morphism $\kappa_f: \Sigma^{\infty}_{(\bb{P}^1,\infty)}(\spc{B})\to M_{(\bb{P}^1,\infty)}(\pair{B})_f$
is a stable motivic equivalence of  $(\bb{P}^1,\infty)$-spectra.

\item The$(\bb{P}^1,\infty)$-spectrum $M_{(\bb{P}^1,\infty)}(\pair{B})_f$ is a motivically fibrant
$\Omega$-spectrum. This means that for each integer
$n \geq 0$, each motivic space $C_*(Fr(-,\pair{B} \wedge \pair{T}^{n}))_f$ is motivically fibrant in the motivic model category of Morel and Voevodsky~\cite{MV} of simplicial pointed Nisnevich sheaves,
and the structural morphism

   \[C_*(Fr(-,\pair{B}\wedge \pair{T}^{n}))_f\to \Omega_{(\bb{P}^1,\infty)}(C_*(Fr(-,\pair{B}\wedge \pair{T}^{\wedge n+1}))_f)\]
	
is a shemewise weak equivalence.

\item The canonical morphism of simplicial pointed Nisnevich sheaves

\[C_*Fr(-,(X,X-S))_f\to \Omega^{\infty}_{(\bb{P}^1,\infty)} \Sigma^{\infty}_{(\bb{P}^1,\infty)} (X/(X-S))\]

is a schemewise homotopy equivalence.
In particular, for any finitely generated field extension $K/k$, the canonical morphism of simplicial sets

\[C_*Fr(Spec(K),(X,X-S))\to \Omega^{\infty}_{(\bb{P}^1,\infty)} \Sigma^{\infty}_{(\bb{P}^1,\infty)} (X/(X-S)) (Spec(K))\]

is a weak equivalence.
\end{enumerate}
\end{theorem}
In addition, one has the following Proposition.
\begin{proposition} \label{p:connectedness}
Let $r>0,$ and let $(X,X-S) \in SmOp(Fr_0(k))$ be such a pair that $codim_{X_i}(S \cap X_i)>r$ in each connected 
(Or, which is the same for a $k$-smooth scheme, irreducible) component $X_i \subseteq X.$ Then the simplicial sheaf $C_*Fr(-,(X,X-S))$
is locally $r$-connected in the Nisnevich topology.
\end{proposition}

Part $(1)$ of Theorem~\ref{Segal_Thm_II}   is proved in~\cite[Subsection 9.2]{GP}, in even greater generality.
Part $(3)$ directly follows from Parts $(1)$ and $(2)$.
Part $(2)$ requires proof in our case and is one of the main results of this paper.

The proof of this theorem uses the theory of framed motives, introduced and developed in~\cite{GP}. According to~\cite[Definition 5.2.(2)]{GP}, the framed motive functor is a functor
$$M_{fr}: \Delta^{op}SmOp(Fr_0(k))\to Sp_{S^1}(k),$$
where $Sp_{S^1}(k)$ is the category of $S^1$-spectra of pointed Nisnevich sheaves on $Sm/k$.
Recall that the framed motive $M_{fr}(\pair{B})$ of the pair $\pair{B}=(X,U)\in SmOp(Fr_0(k))$ is the Segal $S^1$-spectrum
$$M_{fr}(\pair{B})=(C_*Fr(-,\pair{B}),C_*Fr(-,\pair{B}\otimes S^1),C_*Fr(-,\pair{B}\otimes S^2),\ldots),$$
corresponding to the $\Gamma$-space
$K\mapsto C_*Fr(-,\pair{B}\otimes K)=Fr(\Delta^\bullet\times-,\pair{B}\otimes K)$.
The internal meaning of the framed motive $M_{fr}(\pair{B})$ of the pair $\pair{B}=(X,U)\in SmOp(Fr_0(k))$ is that in the motivic category
$SH_{S^1}(k)$ there is a canonical isomorphism
$\Omega^{\infty}_{\mathbb G}\Sigma^\infty_{\bb G}\Sigma^{\infty}_{S^1}(\spc{B}) \to M_{fr}(\pair{B})$
(see \cite[Introduction]{GP}).

\begin{note}\label{recollection}
In the course of proving the Theorem, we will need the following categories, functors and their natural transformations introduced in \cite{GP}:
the symmetric monoidal categories $SmOp(Fr_0(k))$, $\Delta^{\op}SmOp(Fr_0(k))$, the functor
$//: SmOp(Fr_0(k))\to \Delta^{\op}SmOp(Fr_0(k))$
taking $(X,U)$ to $X//U$ and its generalisations taken from \cite[Section 5]{GP} and \cite[Section 8]{GP}.
\end{note}

\begin{definition} \label{definition_//_f}

for any morphism $f:Y\to Z$ in $Fr_0(k)$ denote by
$Z//_f Y$ the simplicial object in the category $Fr_0(k)$ which is the coproduct in $\Delta^{op}Fr_0(k)$
of the diagram

	\[Z\xleftarrow f Y\hookrightarrow Y\otimes I.\]
	
Here for a pointed set $(K,*),$ the object $Y \otimes K\in Fr_0(k)$ is taken to be as in~\cite[раздел 8]{GP}.
For a pointed simplicial set $A_*,$ $Y \otimes A_*$
is taken to be as the simplicial object $(Y \otimes A)_n = Y \otimes A_n.$ 	
(Definition taken taken from~\cite[раздел 8]{GP})
\end{definition}

%\begin{note}
%Можно думать о $Z//_f Y$ как о конусе морфизма $f$ в следующем смысле:
%Если $F: Fr_0(k) \to \mathscr{C}$ --- функтор в некоторую симплициальную модельную категорию
%$\mathscr{C},$ сохраняющий конечные копроизведения и переводящий все объекты в кофибрантные,
%и он продолжается до функтора $\Delta^{op}F: \Delta^{op}Fr_0(k) \to F$
%при помощи функтора геометрической реализации $\Delta^{op}\mathscr{C} \to \mathscr{C},$
%то диаграмма $Y \to Z \to Z//_f Y$ переходит в последовательность кослоя в $\mathscr{C}$ под действием $\Delta^{op}F.$
%\end{note}

\begin{note}
For any $Y,Z\in Fr_0(k)$ one has:
$Z//_f Y$ is a simplicial $k$-smooth scheme.
\end{note}

\begin{note} \label{alpha}
There is a morphism natural in the pair $\pair{B}:=(X,U)\in SmOp(Fr_0(k)),$

 $\alpha_\pair{B}: X//U \to (X,U)=\pair{B}$ to the constant simplicial object
$[n]\mapsto (X,U)=\pair{B}$ in the category $\Delta^{\op}SmOp(Fr_0(k))$; the construction of $\alpha_\pair{B}$  is given in the Introduction to \cite{GNP},
where it is denoted by $\alpha$.
\end{note}

\begin{definition} \label{wedge}

If $X\in Sm/k$ and $x\in X$ is a $k$-rational point, we write$X^{\wedge 1}$ (implying $x$) for
$X//x$. In particular, we use
$\bb G_m^{\wedge 1}$ for $\bb G_m // 1$.

Considering $\Delta^{\op}Fr_0(k)$ as a full subcategory in the symmetric monoidal category
$\Delta^{\op}SmOp(Fr_0(k))$, we can take the $n$th monoidal degree of $X//x$ for each $n>0$.
We denote it by  $X^{\wedge n}$.
An example is $\bb G_m^{\wedge n}$.

\end{definition}

In the course of the proof we will need statements of four kinds:

\begin{enumerate}
\item{ Previously proved statements}
\item{ Statemets with proof transferrable from a previously proved case with insignificant change}
\item{ The local connectedness of spaces of the type $C_* Fr(-,(X,X-S))$}
\item{ The morphism $\alpha_{(X,U)}: X//U \to (X,U)$ in the category $\Delta^{\op}SmOp(Fr_0(k))$ from Note~\ref{alpha} induces a stable local equivalence $M_{fr}(\alpha_{(X,U)}): M_{fr}(X//U)\to M_{fr}(X,U)$ in the category $Sp_{S^1}(k)$.
}
\end{enumerate}
Statements of the third and fourth kind are the main difficulty in proving Theorem~\ref{Segal_Thm_II}.
%Конструкция симплициального конуса $(X,U)\mapsto X//U\in \Delta^{op}(Fr_0/k)$ и более общий ее вариант приведены
%в \cite[Section 8]{GP}.

Below are the statements from goups (2)-(4), as well as several statements from group (1)

\subsection{Group 1}
\begin{definition}\label{frmotive}(=~\cite[Definition 5.2.(1)]{GP})
The framed motive $\cc M_{fr}(\cc G)$ of a simplicial pointed Nisnevich sheaf
$\cc G$ is the Segal $S^1$-spectrum
$(C_*\cc Fr(-,\cc G),C_*\cc Fr(-,\cc G\wedge S^1),C_*\cc Fr(-,\cc G\wedge S^2),\ldots)$
associated with the 
$\Gamma$-space
$K\in\Gamma^{\op}\mapsto C_*\cc Fr(-,\cc G\wedge K)=\cc Fr(\Delta^\bullet_+\wedge-,\cc G\wedge K)$.
\end{definition}

\begin{note}\label{two_fr_motives_are_eq}
The framed motive
$M_{fr}(\pair{B}_{\bullet})$ of the simplicial pair $\pair{B}_{\bullet}\in\Delta^{op}SmOp(Fr_0(k))$
is equal to the framed motive $\cc M_{fr}(spc(\pair{B}_{\bullet}))$
of the motivic space $spc(\pair{B}_{\bullet})$. A priori two functors exist
$$\cc M_{fr}\circ spc, M_{fr}: \Delta^{op}SmOp(Fr_0(k))\rightrightarrows Sp_{S^1}(k).$$
However they coincide owing to Note~\ref{note_lemma_Voev}.
Sometimes one of them is useful, and sometimes the other.
\end{note}

Recall that in~\cite[Section 9.1]{GP} for a finitely presented (defined in the cited section) $A\in sShv_\bullet(Sm/k)$ canonical morphisms are introduced:
   \begin{equation}\label{smashell}
  C_*\cc Fr(L)\to\underline{\Hom}(A,C_*\cc Fr(L\wedge A)) \ \text{and} \  \cc M_{fr}(L)\to\underline{\Hom}(A,\cc M_{fr}(L\wedge A)).
   \end{equation}
They induce morphisms \begin{equation}\label{smashell_2}
   C_*\cc Fr(L)_f\xrightarrow{a_A} \underline{\Hom}(A,C_*\cc Fr(L\wedge A)_f),\  \text{and} \
   \cc M_{fr}(L)_f\xrightarrow{\alpha_A} \underline{\Hom}(A,\cc M_{fr}(L\wedge A)_f),
   \end{equation}
\normalsize
where the meaning of the lower index $f$ for spaces is given in the beginning of this paper, after Notations~\ref{n:general}, and for $S^1$-spectra it means taking a levelwise Nisnevich-local fibrant replacement in the category of $S^1$-spectra.
%$C_*\cc Fr(L)_f,C_*\cc Fr(L)_f$ (respectively $\cc M_{fr}(L)_f,\wt{\cc M}_{fr}(L)_f$) are Nisnevich local
%fibrant replacements of $C_*\cc Fr(L),C_*\cc Fr(L)$ (respectively a level Nisnevich local
%fibrant replacement of $\cc M_{fr}(L),\wt{\cc M}_{fr}(L)$ in the category of ordinary $S^1$-spectra).

\begin{lemma}\label{chempion} (=~\cite[Lemma 9.1]{GP})
Let $u:A\to B$ be a motivic weak equivalence in
$sShv_\bullet(Sm/k)$ between finitely presented objects, such that the induced morphism $u_*:\cc M_{fr}(L\wedge A)\to\cc M_{fr}(L\wedge B)$
is Nisnevich-locally a stable weak equivalense of spectra. Suppose that the fibrant replacements in the levelwise Nisnevich-local injective model structure
$\cc M_{fr}(L)_f,\cc M_{fr}(L\wedge A)_f,\cc M_{fr}(L\wedge B)_f$
are all motivically fibrant $S^1$-spectra. Then $\alpha_A:\cc
M_{fr}(L)_f\to\underline{\Hom}(A,\cc M_{fr}(L\wedge A)_f)$ is a schemewise stable equivalence if and only if 
$\alpha_B:\cc
M_{fr}(L)_f\to\underline{\Hom}(B,\cc M_{fr}(L\wedge B)_f)$
is such.
\end{lemma}

\begin{corollary}\label{cor:A1_fibrant_2} (=~\cite[Corollary 7.5]{GP})
Let $k$ be an infinite perfect field, and let $Y$ be a simplicial object in $SmOp(Fr_0(k))$. Suppose the simplicial pointed Nisnevich sheaf $C_*Fr(Y)$ is Nisnevich-locally connected. Let
$M_{fr}(Y) \to M_{fr}(Y)_f$ be the fibrant replacement in the levelwise Nisnevich-local injective model structure on $S^1$-spectra of pointed simplicial sheaves.
Then:
\begin{enumerate}
\item $M_{fr}(Y)_f$ is fibrant in the stable injective motivic model category of $S^1$-spectra;

\item for any $n\geq 0$ and any fibrant replacement $C_*(Fr(-,Y\otimes
{S}^n))\to C_*(Fr(-,Y\otimes {S}^n))_{f}$ в $sShv_\bullet(Sm/k)$, the space
$C_*(Fr(-,Y\otimes {S}^n))_{f}$ is motivically fibrant in $sShv_\bullet(Sm/k)_{\spc}$.
\end{enumerate}
\end{corollary}

\subsection{Group 2}

Let $\pair{B} \in SmOp(Fr_0(k)).$
Consider the following commutative diagram in $\Delta^{\op}Fr_0(k)$ (= \cite[p. 33 Formula (12)]{GP}

   \begin{equation}\label{ska-vs-dinamo}
    \xymatrix{\bb G_m\ar[r]\ar[d]&\bb A^1\ar[r]\ar[d]&\bb A^1//\bb G_m\ar[d]^{\alpha}\\
               \bb G_m^{\wedge 1}\ar[r]\ar@{=}[d]&\bb A^{\wedge 1}\ar[r]\ar[d]
               &\bb A^{\wedge 1}//\bb G_m^{\wedge 1}\ar[d]^\beta\\
               \bb G_m^{\wedge 1}\ar[r]&\emptyset\ar[r]&\bb G_m^{\wedge 1}\otimes {S}^1}
   \end{equation}
It induces a morphism of framed motives

   \[\beta_*\alpha_*:M_{fr}(\pair{B}\wedge \pair{T}^{\wedge n}\wedge (\A^1//\bb G_m))\to M_{fr}(\pair{B}\wedge \pair{T}^{n}\wedge\bb G_m^{\wedge 1}\otimes {S}^1),\quad n\geq 0,\]

\begin{theorem} \label{AlphaBeta} (cf.~\cite[Theorem 8.2]{GP})
The morphism $\beta_*\alpha_*$ is Nisnevich-locally a weak equivalence of $S^1$-spectra, if $char(k) \neq 2.$
\end{theorem}
\subsection{Group 3}

\begin{definition} \label{definition_connected}
We call a pair $\pair{B}=(X,U) \in SmOp(Fr_0(k))$ {\it connected}, if $U$ itersects each connected (=irreducible) component of $X$ in a nonemty subscheme.
\end{definition}

\begin{note}
In case of $k=\bb{C},$ the pair $(X,U)$ is connected if and only if the pair of topological varieties $(X^{an},U^{an})$ is connected in the general topological sense.
\end{note}

\begin{example}
For any pair $\pair{B}$ and integer $n \geq 1,$ the pair $\pair{B} \wedge \pair{T}^{\wedge n}$ is connected
\end{example}

\begin{proposition}\label{prop:connectedness}
For any integer $n \geq 0$ and any connected pair $\pair{B}=(X,U)\in SmOp(Fr_0(k))$ the spaces
\[C_*Fr(\pair{B}\wedge \pair{T}^{\wedge n}), C_*Fr(\pair{B}\wedge \pair{T}^{\wedge n}\wedge (\A^1//\Gm)), C_*(Fr(\pair{B}\wedge \pair{T}^{\wedge n}\wedge \bb G_m^{\wedge 1}\otimes {S}^1))\]
are Nisnevich-locally connected.
\end{proposition}

Hence we have the following.

\begin{proposition}\label{cor:A1_fibrant_3}
(cf.~\cite[Corollary 7.6]{GP})

Let
%$M_{fr}(Y) \to M_{fr}(Y)_f$ 
$E\mapsto E_{f}$ be a fibrant replacement in the levelwise injective local model structure on the sheaves of $S^1$-spectra.
For each integer $n \geq 0$ and each connected pair $\pair{B}=(X,U)\in SmOp(Fr_0(k)),$ the $S^1$-spectra
\[M_{fr}(\pair{B}\wedge \pair{T}^{\wedge n})_f,
 M_{fr}(\pair{B}\wedge \pair{T}^{\wedge n}\wedge \bb G_m^{\wedge 1}\otimes { S}^1)_f
\text{ and }
M_{fr}(\pair{B}\wedge \pair{T}^{\wedge n}\wedge (\A^1//\Gm))_f\]
are fibrant in the stable injective motivic model category of $S^1$-spectra. In particular, they are fibrant replacements in the stable injective Nisnevich-local model structure.

Furthermore, for connected pairs $\pair{B}$
the motivic spaces
\[C_*Fr(\pair{B} \wedge \pair{T}^{\wedge n})_f,
C_*(Fr(\pair{B}\wedge \pair{T}^{\wedge n}\wedge \bb G_m^{\wedge 1}\otimes {S}^1))_f
\text{ and }
C_*Fr(\pair{B}\wedge \pair{T}^{\wedge n}\wedge (\A^1//\Gm))_f\]
are motivically fibrant in the motivic model structure of Morel and Voevodsky on $sShv_\bullet(Sm/k)_{\spc}$.
\end{proposition}

\subsection{Group 4}

\begin{theorem} \label{Quotient-Cone} (cf.~\cite[Theorem 1.2]{GNP})
Let $char(k) \neq 2.$
Let $S \subset X$ be a closed subscheme of a $k$-smooth scheme.
Put $U=X-S$ and denote the pair $(X,U)$ by $\pair{B} \in SmOp(Fr_0(k))$,
and the pair $(\A^1,\Gm)$ by $\pair{T}$. Consider morphisms $\alpha_\pair{B}: (X//U) \to \pair{B}$ and $\alpha_{\pair{T}}: (\A^1//\Gm) \to \pair{T}$
from Note~\ref{alpha} in the category $\Delta^{op}Fr_0(k)$.
For each $n \geq 0$ consider the morphism
\[\alpha_{\pair{B}} \wedge \alpha^{\wedge n}_{\pair{T}}: (X//U) \wedge (\A^1//\Gm)^{\wedge n} \to \pair{B} \wedge \pair{T}^{\wedge n}.\]

Then the induced morphism
\[ M_{fr}(\alpha_{\pair{B}} \wedge \alpha^{\wedge n}_{\pair{T}}): M_{fr}((X//U) \wedge {(\A^1//\Gm)^{\wedge n}} ) \to M_{fr}(\pair{B} \wedge \pair{T}^{\wedge n})\]
is Nisnevich-locally a stable weak equivalence of $S^1$-spectra.
\end{theorem}

\begin{corollary} \label{T-cone} (cf.~\cite[Corollary 8.1]{GNP})
In the conditions and notations of Theorem \ref{Quotient-Cone}
for each $n\geq 0$ and each $\pair{B} \in SmOp(Fr_0(k))$ the morphism
\[M_{fr}(id_\pair{B}\wedge id^{\wedge n}_{\pair{T}}\wedge \alpha_{\pair{T}}):  
M_{fr}(\pair{B}\wedge \pair{T}^{\wedge n}\wedge(\A^1//\Gm))\to M_{fr}(\pair{B}\wedge \pair{T}^{\wedge n+1})\]
is Nisnevich-locally a stable weak equivalence of $S^1$-spectra.
\end{corollary}

\subsection{Reduction of the main Theorem}

\begin{proof}[Proof of Theorem~\ref{Segal_Thm_II} modulo several statements]
Note that the pair $\pair{B}$ of Theorem~\ref{Segal_Thm_II} is connected in the sense of Definition~\ref{definition_connected}.
By Proposition~\ref{cor:A1_fibrant_3}, for each integer $n \geq 0$, the
 $S^1$-spectrum $M_{fr}(\pair{B}\wedge \pair{T}^{\wedge n})_f$ is motivically fibrant, and the motivic space
 $C_*Fr(\pair{B}\wedge \pair{T}^{\wedge n})_f$ is motivically fibrant.
Let $u:(\bb{P}^1,\infty)\to T$ be a canonical motivic weak equivalence in $sShv_\bullet(Sm/k)$.
Consider the induced morphism of motivic spaces
\[u^*: \underline{\Hom}(T,C_*Fr(\pair{B}\wedge \pair{T}^{\wedge n+1})_f)\to \underline{\Hom}((\bb{P}^1,\infty),C_*Fr(\pair{B}\wedge \pair{T}^{\wedge n+1})_f).\]
%Она также задается фрейм-соответствием
%$(\{0\},\bb A^1,t)\in Fr_1(pt,pt)$.
Now use the morphisms $a_T$ and $\alpha_T$ defined before Lemma \ref{chempion}.
Specifically, by~\cite[Lemma 9.3]{GP}, the morphism
  \[u^*\circ a_T:  C_*Fr(\pair{B}\wedge \pair{T}^{\wedge n})_f\to\underline{\Hom}((\bb{P}^1,\infty),C_*Fr(\pair{B}\wedge \pair{T}^{\wedge n+1})_f)\]
is a schemewise weak equivalence if and only if the morphism
$\alpha_T:M_{fr}(\pair{B}\wedge \pair{T}^{\wedge n})_f\to\underline{\Hom}(T,M_{fr}(\pair{B}\wedge \pair{T}^{\wedge n+1})_f)$
is a schemewise stable equivalence of spectra.

Consider the following roof in the category $\Delta^{op}SmOp(Fr_0(k)):$
   \[\pair{T}\xleftarrow{} (\A^1//\Gm)\xrightarrow{}\bb G_m^{\wedge 1}\otimes { S}^1,\]
where the right attow is equal to $\beta\alpha$, defined in \eqref{ska-vs-dinamo}.
Its arrows become motivic weak equivalences after applying the functor $spc.$

by Proposition~\ref{cor:A1_fibrant_3}, for each integer $n\geq 0$, the $S^1$-spectra
\[M_{fr}(\pair{B}\wedge
\pair{T}^{\wedge n}\wedge \bb G_m^{\wedge 1}\otimes { S}^1)_f,  M_{fr}(\pair{B}\wedge
\pair{T}^{\wedge n}\wedge (\A^1//\Gm))_f\]
are motivically fibrant, and
\[C_*(Fr(\pair{B}\wedge \pair{T}^{\wedge n}\wedge \bb G_m^{\wedge 1}\otimes {S}^1))_f, C_*Fr(\pair{B}\wedge \pair{T}^{\wedge n}\wedge (\A^1//\Gm))_f\]
are motivically fibrant spaces.
By Corollary~\ref{T-cone}
\[M_{fr}(\pair{B}\wedge \pair{T}^{\wedge n}\wedge (\A^1//\Gm)) \to M_{fr}(\pair{B}\wedge \pair{T}^{\wedge n+1})\]
is Nisnevich-locally a stable weak equivalence of spectra. By Theorem~\ref{AlphaBeta},
\[M_{fr}(\pair{B}\wedge \pair{T}^{\wedge n}\wedge (\A^1//\Gm))\to M_{fr}(\pair{B}\wedge \pair{T}^{\wedge n}\wedge\Gm^{\wedge 1}\otimes {S}^1)\]
is a Nisnevich-local stable weak equivalence of spectra. By Lemma~\ref{chempion},
	\[\alpha_T:M_{fr}(\pair{B}\wedge \pair{T}^{\wedge n})_f\to\underline{\Hom}(T,M_{fr}(\pair{B}\wedge \pair{T}^{\wedge n+1})_f)\]
is a schemewise stable equivalence of spectra if and only if the morphism of spectra
	\[\alpha_{\Gm^{\wedge 1}\otimes {S}^1}:M_{fr}(\pair{B}\wedge \pair{T}^{\wedge n})_f\to \underline{\Hom}((\Gm^{\wedge 1}\otimes {S}^1),M_{fr}(\pair{B}\wedge \pair{T}^{n}\wedge\Gm^{\wedge 1}\otimes {S}^1)_f).\]
is a schemewise stable equivalence of spectra.
We temporarily denote $X//U$ by $c(\pair{B})$, $\A^1//\Gm$ by
$c(\pair{T})$, and consider the following commutative diagram:

\[\begin{tikzcd}
{M_{fr}(c(\pair{B})\wedge c(\pair{T}) ^{\wedge n})_f} \arrow[r,"{\alpha_{\Gm^{\wedge 1}\otimes {S}^1}}"] \arrow[d]
& \underline{\Hom}((\Gm^{\wedge 1}\otimes {S}^1),M_{fr}(c(\pair{B})\wedge c(\pair{T})^{\wedge n}\wedge \Gm^{\wedge 1}\otimes {S}^1)_f) \arrow[d] \\
M_{fr}(\pair{B}\wedge \pair{T}^{\wedge n})_f\arrow[r,"{\alpha_{\bb G_m^{\wedge 1}\otimes {S}^1}}"]
& \underline{\Hom}((\bb G_m^{\wedge 1}\otimes {S}^1),M_{fr}(\pair{B}\wedge \pair{T}^{n}\wedge \bb G_m^{\wedge 1}\otimes {S}^1)_f)
\end{tikzcd}\]

where the lower index $f$ denotes passing
%$M_{fr}(c(\pair{B})\wedge (\A^1//\Gm) ^{\wedge n})_f$
to the stable Nisnevich-local fibrant replacement of spectra, and
$c(\pair{T})^{\wedge n}=(\A^1//\Gm)^{\wedge n} \in \Delta^{op}Fr_0(k)$ was constructed in Definition~\ref{wedge}.

It follows from Theorem~\ref{Quotient-Cone}, and Note~\ref{n:schemewise} below that the left vertical arrow is a schemewise stable equivalence of spectra, hence the right vertical arrow is such. Therefore, the bottom arrow is a schemewise stable equivalence of spectra if and only if the top arrow is such. But that is a schemewise stable weak equivalence by the Cancellation theorem for framed motives of algebraic varieties~\cite[Theorem~A]{AGP} combined with \cite[Теорема 6.5]{GP}.
To conclude, the morphism $\alpha_T$, and with it the morphism $u^*\circ a_T,$ are motivic equivalences.
\end{proof}

\begin{note} \label{n:schemewise}
Let $A_f \to B_f$ be a weak equivalence between fibrant obects in the stable injective local model structure on $S^1$-spectra of simplicial sheaves on $Sm/k.$ Then it is a schemewise stable equivalence of $S^1$-spectra. 

Indeed, a simple lifting property check shows that the homotopy fiber of the morphism $A_f \to B_f$ is fibrant. It is also (locally) weakly equivalent to the point $*.$ We apply the lifting property of this weak fibration to cofibrations $B \otimes S^m \to B \otimes (S^m \wedge I_*)$ where $I_*$ is the pointed segment. The lifting property shows that the homotopy fiber space has no homotopy groups over an arbitrary scheme $B.$
\end{note}

Below we prove the statements from groups (3) and (4).

\section{Reducing Theorem~\ref{Quotient-Cone} to the linear Cone theorem}

The following is a generalisation of \cite[Theorem 1.2]{GNP}.

\begin{theorem} \label{ZM_fr_and_LM_fr}
Assume $char(k) \neq 2.$
For any pair $\pair{B} \in SmOp(Fr_0(k)),$ the following natural morphism of framed $S^1$-spectra is a schemewise stable equivalence:

\[\lambda_{\pair{B}}:\bb Z\Fr_*^{S^1}(\pair{B})\to EM(\ZF_*(-,\pair{B}))\]

Furthermore, the natural morphism of framed $S^1$-spectra

\[l_{\pair{B}}: \bb ZM_{fr}(\pair{B})\to LM_{fr}(\pair{B})\]

is a schemewise stable equivalence.  In particular, for each $U\in Sm/k,$ 

\[\pi_*(\bb ZM_{fr}(\pair{B})(U))=H_*(\ZF(\Delta^\bullet\wedge U,\pair{B}))=H_*(C_*\bb Z\F(U,\pair{B})).\]

\end{theorem}
The notations used in this theorem are introduced analogously to \cite[Раздел 8]{GNP}

\begin{proof}
The proof in \cite[Appendix B]{GNP} translates without change to a proof for an arbitrary smooth open pair $\pair{B}$ in place of $X \wedge \pair{T}^{\wedge n}.$
\end{proof}

Reduce Theorem~\ref{Quotient-Cone} to Theorem~\ref{genCone} below.

\begin{proof}
Prove via induction in $n$. The base is $n=0.$
By Theorem~\ref{genCone}, the morphism 
\begin{equation}\label{eq: C_*_X/X-S}
C_*\bb Z\F(X)/C_*\bb Z\F(X-S) \to C_*\bb Z\F(\pair{B})
\end{equation}
of sheaves of abelian groups is a loval quasiisomorphism.

From the beginning of~\cite[Section 8]{GNP} (or, in the case of the last spectrum, by an analogous argument), we know that the $S^1$-spectra

\[LM_{fr}(X), LM_{fr}(X-S) \text{ and } LM_{fr}(\pair{B})\]

 are Eilenberg-MacLane $S^1$-spectra of complexes $C_*\bb Z\F(X)$, $C_*\bb
Z\F(X-S)$ and $C_*\bb Z\F(\pair{B})$ respectively. Therefore the morphism
\begin{equation*}\label{eq: LM_fr_X/X-S}
LM_{fr}(X)/LM_{fr}(X-S)
\to LM_{fr}(\pair{B})
\end{equation*}
induced by~\eqref{eq: C_*_X/X-S} is a stable weak equivalence, hence the morphism

   \[\bb ZM_{fr}(X)/\bb ZM_{fr}(X-S) \to \bb ZM_{fr}(\pair{B}),\]
	
also is a stable weak equivalence	
by Theorem~\ref{ZM_fr_and_LM_fr}. The $S^1$-spectra $M_{fr}(X)$, $M_{fr}(X-S)$, $M_{fr}(\pair{B})$ are
$(-1)$-connected, since they are Segal spectra (see \cite[Definition 5.2]{GP}, \cite[Proposition 1.4]{Seg})

The stable Whitehead Theorem~\cite[II.6.30]{Sch}
implies that the morphism
\begin{equation}
M_{fr}(X)/M_{fr}(X-S) \to M_{fr}(\pair{B})
\end{equation}
is a local stable weak equivalence.

The induction base is thus proved, and we move on to proving the induction step $n \to n+1$.

\begin{equation}\label{eq: C_*_A1/Gm}
C_*\bb Z\F(\pair{B} \wedge \pair{T}^{\wedge n} \wedge \A^1)/C_*\bb Z\F(\pair{B} \wedge \pair{T}^{\wedge n}\wedge  \Gm) \to C_*\bb Z\F(\pair{B} \wedge \pair{T}^{\wedge n+1})
\end{equation}
is a local quasiisomorphism, by Thoeorem~\ref{genCone}.

By the same argument as given in the beginning of~\cite[Section 8]{GNP}, the $S^1$-spectra

\[LM_{fr}(\pair{B} \wedge \pair{T}^{\wedge n}\wedge \A^1), LM_{fr}(\pair{B} \wedge \pair{T}^{\wedge n}\wedge  \Gm) \text{ and } LM_{fr}(\pair{B} \wedge \pair{T}^{\wedge n+1})\]

 are Eilenberg-MacLane $S^1$-spectra of complexes

\[C_*\bb Z\F(\pair{B} \wedge \pair{T}^{\wedge n}\wedge \A^1), C_*\bb Z\F(\pair{B} \wedge \pair{T}^{\wedge n}\wedge  \Gm) \text{ and } C_*\bb Z\F(\pair{B} \wedge \pair{T}^{\wedge n+1})\]

 respectively. Therefore the morphismm

\begin{equation*}\label{eq: LM_fr_A1/Gm}
LM_{fr}(\pair{B} \wedge \pair{T}^{\wedge n}\wedge \A^1)/LM_{fr}(\pair{B} \wedge \pair{T}^{\wedge n}\wedge  \Gm)
\to LM_{fr}(\pair{B} \wedge \pair{T}^{\wedge n+1})
\end{equation*}

induced by~\eqref{eq: C_*_A1/Gm} is a local stable weak equivalence, therefore the morphism

   \[\bb ZM_{fr}(\pair{B} \wedge \pair{T}^{\wedge n}\wedge \A^1)/\bb ZM_{fr}(\pair{B} \wedge \pair{T}^{\wedge n}\wedge  \Gm) \to \bb ZM_{fr}(\pair{B} \wedge \pair{T}^{\wedge n+1})\]
	
	is also such,
by Theorem~\ref{ZM_fr_and_LM_fr}.

The $S^1$-spectra $M_{fr}(\pair{B} \wedge \pair{T}^{\wedge n}\wedge
\A^1)$, $M_{fr}(\pair{B} \wedge \pair{T}^{\wedge n}\wedge \Gm)$, $M_{fr}(\pair{B} \wedge \pair{T}^{\wedge n+1})$
are $(-1)$-connected, since they are Segal spectra (see \cite[Definition 5.2]{GP}, \cite[Proposition 1.4]{Seg}).

By the stable Whitehead theorem~\cite[II.6.30]{Sch}, the morphism

\begin{equation}\label{eq: M_fr_A1/Gm}
M_{fr}(\pair{B} \wedge \pair{T}^{\wedge n}\wedge \A^1)/M_{fr}(\pair{B} \wedge \pair{T}^{\wedge n}\wedge  \Gm) \to M_{fr}(\pair{B} \wedge \pair{T}^{\wedge n+1})
\end{equation}

is a local stable weak equivalence. Consider the following sequence of natural morphisms
\begin{multline*}
M_{fr}((X//U) \wedge (\A^1//\Gm)^{\wedge n+1})=M_{fr}((X//U)\wedge (\A^1//\Gm)^{\wedge n}\wedge (\A^1//\Gm))\xrightarrow{(1)} \\
Cone[M_{fr}((X//U)\wedge (\A^1//\Gm)^{\wedge n}\wedge \Gm ) \to M_{fr}((X//U)\wedge (\A^1//\Gm)^{\wedge n}\wedge \A^1)] \xrightarrow{(2)} \\
Cone[M_{fr}(\pair{B}\wedge \pair{T}^{\wedge n}\wedge \Gm) \to M_{fr}(\pair{B}\wedge \pair{T}^{\wedge n}\wedge \A^1)] \xrightarrow{(3)} \\
M_{fr}(\pair{B}\wedge \pair{T}^{\wedge n}\wedge \A^1)/M_{fr}(\pair{B}\wedge \pair{T}^{\wedge n}\wedge \Gm) \xrightarrow{(4)}
M_{fr}(\pair{B}\wedge \pair{T}^{\wedge n+1}).
\end{multline*}
Arrows $(1)$ and $(3)$ are schemewise stable weak equivalences by standard arguments. Arrow $(2)$ is a local stable weak equivalence by the induction hypothesis. Arrow $(4)$ is exactly the morphism~\eqref{eq: M_fr_A1/Gm}, hence it is a local stable weak equivalence.

Hence for all $\ell\geq 0,$ the canonical morphism
\begin{equation}\label{eq:cone_and_T}
M_{fr}((X//U)\wedge (\A^1//\Gm)^{\wedge\ell})\to
M_{fr}(\pair{B}\wedge \pair{T}^{\ell})
\end{equation}
is a local stable weak equivalence.

\end{proof}

Similar to~\cite[Следствие 8.1]{GNP}, we deduce from this theorem Corollary~\ref{T-cone} (we actually use a variant of the theorem for simplicial pairs).

\section{The linear cone theorem}

In this section a plan of proof of the following theorem is laid out:

\begin{theorem} \label{genCone}

Let $\pair B=(X,U),\pair M=(X',U')$ be pairs, i.e. objects of $SmOp(Fr_0(k).$
Then the sequence 

\[ С_* \ZF(-,\pair M \wedge U) \to С_* \ZF(-,\pair M \wedge X) \to С_* \ZF(-,\pair M \wedge \pair B) \]

can be included in a distinguished triangle in the derived category of Nisnevich sheaves.
\end{theorem}

\begin{proof}
We reduce the Theorem to Proposition~\ref{p:QAff_Cone} below. Proof of the Proposition will be completed in Section~\ref{Moving}.

Let $X= \bigcup X_i$ be an affine cover of the scheme $X.$ Enumerate the cover elements with ordinals. We prove the Theorem for $\pair B=(\bigcup \limits_{i \leq k} X_i,U \cap \bigcup \limits_{i \leq k} X_i), \pair M=(X',U')$ by transfinite induction on $k.$ 

Since for sheaves to be equal it suffices to show equality on affine smooth schemes (which are quasicompact), the transfinite induction step is achieved by passing to a filtered colimit of complexes of abelian groups (which is an exact functor). The base of induction $k=0$ is the statement of Proposition~\ref{p:QAff_Cone}. We now prove the induction step $k \to k+1.$

We adopt temporary brief notations

\[F_k =\bigcup \limits_{i \leq k} X_i, C_{\bullet}(\pair B)=С_* \ZF(-,Q \wedge \pair B).\]

We then have a commutative diagram

\[\begin{tikzcd}
 C_{\bullet} (U \cap F_k \cap X_{k+1}) \arrow[r] \arrow[d] & C_{\bullet}(F_k \cap X_{k+1}) \arrow[r] \arrow[d] & C_{\bullet} (F_k \cap X_{k+1}, U \cap F_k \cap X_{k+1}) \arrow[d] \\
 C_{\bullet} (U \cap F_k ) \oplus C_{\bullet} (U \cap X_{k+1} ) \arrow[r] \arrow[d] & C_{\bullet}(F_k ) \oplus C_{\bullet}(X_{k+1} )\arrow[r] \arrow[d] & C_{\bullet}(F_k, U \cap F_k) \oplus  C_{\bullet}(X_{k+1}, U \cap X_{k+1})\arrow[d] \\
 C_{\bullet}(U \cap F_{k+1}) \arrow[r] & C_{\bullet}(F_{k+1}) \arrow[r] & C_{\bullet}(F_{k+1}, U \cap F_{k+1}) 
\end{tikzcd}.\]

Its columns are Meyer-Vietoris triangles, which are distinguished by Proposition~\ref{p:MV}. The first line is destinguished by Proposition~\ref{p:QAff_Cone} applied to the quasiaffine scheme $F_k \cap X_{k+1}.$ The second line is a direct sum of a thiangle distinguished by the induction hypothesis and a triangle distinguished by Proposition~\ref{p:QAff_Cone} applied to the affine scheme $X_{k+1}.$ Consequently, the third line is also distinguished, thus the induction step $k \to k+1$ is proved.

\end{proof}

\begin{proposition}\label{p:QAff_Cone}

Let $\pair B=(X,U),\pair M=(X',U')$ be pairs, i.e. objects of $SmOp(Fr_0(k).$

Suppose that the scheme $X$ is quasiaffine.
 
Then the sequence 

\[ С_* \ZF(-,\pair M \wedge U) \to С_* \ZF(-,\pair M \wedge X) \to С_* \ZF(-,\pair M \wedge \pair B) \]

is included in a distinguished triangle in the derived category of Nisnevich sheaves of abelian groups.
\end{proposition}

In what follows (until Section~\ref{s:connectedness}) we suppose $X$ quasiaffine.

Denote $S=X-U, S'=X'-U'.$
$S=X-U$ can be given as a zero set of some functions $f_1, \cdots, f_n \in k[X].$ Fix such a set of functions. Our objects now satisfy the following conditions:

\begin{conditions}\label{QAff}
$\pair B=(X,U),\pair M=(X',U') \in SmOp(Fr_0(k))$ are pairs of smooth schemes; $S=X-U, S'=X'-U'$ (not necessarily smooth); such $f_1,\ldots f_n \in k[X]$ are chosen that $S=V(f_1,\ldots,f_n).$
\end{conditions}

\begin{definition} \label{d:qf,k}
Under Conditions~\ref{QAff},
$Fr_m^{qf,k}(B,\pair M,\pair B;f_1, \cdots f_n)$ is the subset in $Fr_m(B,\pair M \wedge \pair B)$ consisting of classes containing such explicit correspondences $(Z,W,\phi,g)$ that $\{\phi=0\} \cap g^{-1}(S' \times \{f_1=0, \cdots f_k=0\}) \subset W$ is quasifinite over $B.$ 
\end{definition}  

\begin{note}
The property of the explicit framed correspondence described in Definition~\ref{d:qf,k} above is not respected by equivalence of framed correspondences. Refining $W$ can makes the set $\{\phi=0\} \cap g^{-1}(S' \times \{f_1=0, \cdots f_k=0\}) \subset W$ smaller, and it can become quasifinite over $B$, even if it was not before. However, once it is quasifinite, refining $W$ will not change that property. 
\end{note}
 
From this, translating the analogous constructions from framed correspondences, one can introduce

\begin{itemize}
\item{$Fr^{qf,k}(B,\pair M,\pair B;f_1, \cdots f_n)$}
\item{$F_m^{qf,k}(B,\pair M,\pair B;f_1, \cdots f_n)$}
\item{$F^{qf,k}(B,\pair M,\pair B;f_1, \cdots f_n)$}
\item{$\ZF_m^{qf,k}(B,\pair M,\pair B;f_1, \cdots f_n)$}
\item{$\ZF^{qf,k}(B,\pair M,\pair B;f_1, \cdots f_n)$}
\item{$C_*\ZF^{qf,k}(B,\pair M,\pair B;f_1, \cdots f_n)$}

\end{itemize}

It is clear from the definition that $Fr_m^{qf,n}(B, \pair M \wedge \pair B;f_1, \cdots f_n) = Fr_m(B,\pair M \wedge \pair B).$  

$Fr_m^{qf,0}(B,\pair M,\pair B;f_1, \cdots f_n)$ is independent of $f_1, \cdots f_n,$ it only depends on their common zeros $S,$ thus we give it a separate name

\begin{definition}
$Fr_m^{qf}(B,\pair M,\pair B)$ is the subset of $Fr_m(B,\pair M \wedge \pair B)$ consisting of classes containing such correspondences $(Z,W,\phi,g)$ that $\{\phi=0\} \cap g^{-1}(S' \times X) \subset W$ is quasifinite over $B.$ 
\end{definition}

The proof of Proposition~\ref{p:QAff_Cone} splits into two steps.

\begin{proof}
The statement of the Proposition follows from Theorems~\ref{th:loc.exact} and~\ref{th:moving} below.
\end{proof}

\begin{theorem} \label{th:loc.exact}

Let $\pair B=(X,U),\pair M=(X',U')$ be pairs, i.e. objects of $SmOp(Fr_0(k).$
Then the sequence 

\[ С_*\ZF(-,\pair M \wedge U) \to С_*\ZF(-,\pair M \wedge X) \to С_*\ZF^{qf}(-,\pair M,\pair B) \]

is Nisnevich-locally exact.
\end{theorem}

\begin{proof}
The proof is carried over without change from the case $\pair B=(\A^1,\bb G_m)$ proved in~\cite[Section 5]{GNP}.
\end{proof}

\begin{theorem} \label{th:moving}
Under Conditions~\ref{QAff}, let $B$ be a $k$-smooth affine scheme. If $i$ is the embedding $Fr^{qf}(B,\pair M,\pair B) \hookrightarrow Fr(B,\pair M \wedge \pair B),$ then

\[C_* \bb Z i : C_* \ZF^{qf}(B,\pair M,\pair B)  \hookrightarrow C_* \ZF(B,\pair M \wedge \pair B)\]

is a quasiisomorphism.
\end{theorem}

The sets $Fr^{qf,k}(B,\pair M,\pair B;f_1,\cdots,f_n)$ we defined allow us to split $i$ into a composition of $n$ embeddings $i_k:Fr^{qf,k-1}(B,\pair M,\pair B;f_1,\cdots,f_n) \hookrightarrow Fr^{qf,k}(B,\pair M,\pair B;f_1,\cdots,f_n).$ Hence it suffices to show that $C_* \bb Z i_k $ is a quasiisomorphism. That is proved in Section~\ref{Moving}.

\section{Filtrations $(\ZF_m^{qf,k})^{<d}(B,\pair M,\pair B;f_1, \cdots f_n)$ and $(\ZF_m^{qf,k-1})^{<d,k}(B,\pair M,\pair B;f_1, \cdots f_n)$ }

\begin{definition}
Let $B \in Sm(k)$ be affine, and $c=(Z,W,\phi,g) \in Fr_m^{qf,k}(B,\pair M,\pair B;f_1, \cdots f_n).$ The set of polynomials $F_1, \cdots F_r \in k[B] \times \A^{m+1}$ is called \textit{$c$-defining}, if for each point $b \in B$ there is $i \in \{1,\cdots,r\}$ such that:
\begin{itemize}
\item{$F_i(-,u) \neq 0 \in k(u)[\A^{n+1}]$}
\item{$(\phi,f_k \circ pr_2 \circ g) (W_u) \subseteq Z(F_i(-,u))$}
\end{itemize}
\end{definition}

\begin{lemma}
(cf.~\cite[Corollary 6.3]{GNP}) 

Suppose $B \in Sm(k)$ affine, and $c=(Z,W,\phi,g) \in Fr_m^{qf,k}(B,\pair M,\pair B;f_1, \cdots f_n).$ Then there exists a $c$-defining set $F_1, \cdots F_r.$

Moreover, if $f:B' \to B$ is a morphism of affine $k$-smooth schemes, then $f^*F_1, \cdots f^*F_r$ is a $f^*(c)$-defining set.
\end{lemma}

\begin{proof}
We get the first statement is by applying~\cite[Lemma 6.2]{GNP} to the case $Y=B, n=0, \psi = (\phi,f_k \circ pr_2 \circ g).$

The second statement is obvious from the definitions.
\end{proof}

\begin{definition}
For $d \in \bb N,$ $(Fr_m^{qf,k})^{<d}(B,\pair M,\pair B;f_1, \cdots f_n)$ is a subset in  $Fr_m^{qf,k}(B,\pair M,\pair B;f_1, \cdots f_n)$ consisting of such correspondences $c$ wor which there exists a $c$-defining set, in which the degrees of all polynomials $F_i$ are less than $d.$ 

$(Fr_m^{qf,k-1})^{<d,k}(B,\pair M,\pair B;f_1, \cdots f_n)$ is the subset 
\[Fr_m^{qf,k-1}(B,\pair M,\pair B;f_1, \cdots f_n) \cap (Fr_m^{qf,k}(B,\pair M,\pair B;f_1, \cdots f_n))^{<d}.\]
\end{definition}

\begin{lemma}\label{l:exhausting_1}
For all $m,k\geq 0$ and $d>0,$ the following holds:
\begin{itemize}
\item[(i)]
$(Fr_m^{qf,k})^{<d}(-,\pair M,\pair B;f_1, \cdots f_n)$ is a subpresheaf on $AffSm/k$ of the presheaf $Fr_m^{qf,k}(-,\pair M,\pair B;f_1, \cdots f_n)$;
\item[(ii)]
The ascending filtration of the presheaf $Fr_m^{qf,k}(-,\pair M,\pair B;f_1, \cdots f_n)|_{AffSm/k}$  by subpresheaves $(Fr_m^{qf,k})^{<d}(-,\pair M,\pair B;f_1, \cdots f_n)$ is exhausting;
\item[(iii)]
$(Fr_m^{qf,k-1})^{<d,k}(-,\pair M,\pair B;f_1, \cdots f_n)$ is a presheaf on $AffSm/k$ of the presheaf $Fr_m^{qf,k-1}(-,\pair M,\pair B;f_1, \cdots f_n)$;
\item[(iv)]
the ascending filtration of the presheaf $Fr_m^{qf,k-1}(-,\pair M,\pair B;f_1, \cdots f_n)|_{AffSm/k}$ by subpresheaves $(Fr_m^{qf,k-1})^{<d,k}(-,\pair M,\pair B;f_1, \cdots f_n)$ is exhausting.
\end{itemize}
\end{lemma}

Subgroups 

\[(\ZF_m^{qf,k})^{<d}(B,\pair M,\pair B;f_1, \cdots f_n) \leq \ZF_m^{qf,k}(B,\pair M,\pair B;f_1, \cdots f_n)\]

and

\[(\ZF_m^{qf,k-1})^{<d,k}(B,\pair M,\pair B;f_1, \cdots f_n) \leq \ZF_m^{qf,k-1}(B,\pair M,\pair B;f_1, \cdots f_n)\]

are defined as expected.

\begin{corollary}\label{cor:exhausting_2}
The following equalities of presheaves on $AffSm/k$ are true:

   \[\colim_{d} (\ZF_m^{qf,k})^{<d}(-,\pair M,\pair B;f_1, \cdots f_n) =\ZF_m^{qf,k}(-,\pair M,\pair B;f_1, \cdots f_n)\]

and

   \[\colim_{d}(\ZF_m^{qf,k-1})^{<d,k}(-,\pair M,\pair B;f_1, \cdots f_n) =\ZF_m^{qf,k-1}(-,\pair M,\pair B;f_1, \cdots f_n).\]
	
\end{corollary}

\section{A moving lemma} \label{Moving}

\begin{remark} \label{r:quasi-finite} 
Let $c=(Z,W,\phi;g) \in Fr_m^{qf,k}(B, \pair M \wedge \pair B).$ Denote by $\pi$ the composition $W \to \A^m_B \to B.$ Then the morphism 

\[(\pi, \phi,f_1 \circ pr_2 \circ g,\cdots,f_k\circ pr_2 \circ g)|_{g^{-1}(S' \times X}: g^{-1}(S' \times X) \to B \times \A^{m+k}\]
 
is quasifinite over $B \times 0,$ thus, by semicontinuity of fiber dimension on the source, after refining $W$ one can suppose that morphism quasifinite. In what follows we suppose it is.
\end{remark}

Let $c \in Fr_m^{qf,k}(B,\pair M,\pair B;f_1,\cdots,f_n).$ Define the homotopy $h_d^k(c) \in Fr_m (B \times \A^1,\pair M,\pair B;f_1,\cdots f_n):$

\[h_d^k(c):=(Z \times \A^1,W \times \A^1,\phi_1 -s(f_k \circ pr_2 \circ g)^d,\phi_2 -s(f_k \circ pr_2 \circ g)^{d^2}, \cdots \phi_m -s(f_k \circ pr_2 \circ g)^{d^m};g)\]

Denote $t_d^k=h_d^k \circ in_1,$ where $in_1$ is the embedding of $1$ in $\A^1.$

\begin{lemma}\label{l:homotopy_properties}
If $B\in AffSm/k,$ the following holds:
\begin{itemize}
\item [(i)] If $c \in (Fr_m^{qf,k})^{<d}(B,\pair M,\pair B;f_1, \cdots f_n)$, then $t_d^k(c) \in Fr_m^{qf,k-1}(B,\pair M,\pair B;f_1, \cdots f_n)$;
\item [(ii)] If $c \in (Fr_m^{qf,k-1})^{<d,k}(B,\pair M,\pair B;f_1, \cdots f_n)$, then $h_d^k(c) \in Fr_m^{qf,k-1}(B\times \A^1,\pair M,\pair B;f_1, \cdots f_n)$.
\end{itemize}
\end{lemma}

\begin{proof}
We begin with the first statement. Let $c=(Z,W,\phi;g)\in (Fr_m^{qf,k})^{<d}(B,\pair M,\pair B;f_1, \cdots f_n)$. Let $F_1,\ldots F_r
\in k[\A^{m+1}_U]$  be a $c$-defining set with $\deg F_i<d$ for each $i=1,\ldots,r$. We need to check that $t_d^k(c)$ lies in
$Fr_m^{qf,k-1}(B,\pair M,\pair B;f_1, \cdots f_n)$. Denote for brevity $\tilde{f}_k=f_k  \circ pr_2 \circ g.$ We set

\[\begin{tikzcd}
Y \arrow[d,equals] \\
Z(\phi_1 -(\tilde{f}_k)^d,\phi_2 -(\tilde{f}_k)^{d^2}, \ldots \phi_m -(\tilde{f}_k)^{d^m},f_1 \circ pr_2 \circ g,\ldots,f_{k-1} \circ pr_2 \circ g) \cap g^{-1}(S' \times X) \arrow[d, hookrightarrow] \\
 W
\end{tikzcd}\]

 (the hooked arrow here stands for being a subset). We need to check that for each point $b\in B$ the fiber $Y(b)$ of $Y$
over $b$ is finite. Let $\theta: \A^1\to \A^{m+1}$ be the morphism sending $t$ to $(t^d,t^{d^2},...,t^{d^m},t)$.
This is a closed embedding with some image $C=\theta(\A^1)$.

By Note~\ref{r:quasi-finite}, the morphism
 
\[\psi=(\pi, \phi,f_1 \circ pr_2 \circ g,\cdots,f_k \circ pr_2 \circ g)|_{g^{-1}(S' \times X)}\]

is quasifinite. For a point $b\in B$ there is a polynomial $F$ in the $c$-defining set, such that $F(-,u)$ is nonzero, and its zeros $Z(F(-,u))$ in $\A^{m+1}_u$
contain $(\phi,f_k \circ pr_2 \circ g)(W(u)).$ The latter condition is equivalent to requiring that for the polynomial $\widetilde{F} \in k[x_1,\dots x_{m+k}],$ defined as 

\[\widetilde{F}(x_1,\dots x_{m+k})=F(x_1, \dots,x_m,x_{m+k}),\]

\[Z(\widetilde{F}) \supseteq \psi \big ( W(b) \big ).\]

It is clear that $Y(b)$ is contained in

   \[\psi^{-1}(Z(F(-,b))\cap C).\]
	
$Z(F(-,b))\cap C$ is isomorphic to the zero scheme of the polynomial $F(t^d,t^{d^2},...,t^{d^{m+n}},t)$ on the line
$\A^1$  with coordinate $t$. By~\cite[Lemma 7.1]{GNP} the set $Z(F(-,b))\cap C$ is finite, therefore so is $Y(b)$, proving the first statement.

Check the second statement. Let

 \[c=(Z,W,\phi;g) \in (Fr_m^{qf,k-1})^{<d,k}(B,\pair M,\pair B;f_1, \cdots f_n),\]

 and let $F_1,\ldots F_r \in
k[\A^{m+n+1}_U]$  be a $c$-defining set with $\deg F_i<d$ for each 
$i=1,\ldots,r$. We need to check that $h_d^k(c)$ lies in $Fr_m^{qf,k-1}(B\times \A^1,\pair M,\pair B;f_1, \cdots f_n)$. Denote for brevity $\tilde{f}_k=f_k  \circ pr_2 \circ g.$ Set

   \[\begin{tikzcd}
	Y_s \arrow[d,equals]\\
	Z(\phi_1 -s(\tilde{f}_k)^d,\phi_2 -s(\tilde{f}_k)^{d^2}, \ldots \phi_m -s(\tilde{f}_k)^{d^m},f_1,\ldots,f_{k-1}) \cap g^{-1}(S' \times X) \arrow[d, hookrightarrow] \\
	W\times \A^1, 
	\end{tikzcd}\]
	
where $s$  is the coordinate on the added $\A^1$ factor. (the hooked arrow here stands for being a subset) We need to check that for each point $v\in B\times \A^1$ the fiber $Y_s(v)$ $Y_s$ over $v$ is finite. Denote the image of the coordinate dunction $s$ in $k(v)$ by $a.$ 

Let $\theta_a: \A^1_{k(v)}\to \A^{m+1}_{k(v)}$ be the morphism sending $t$ to
$(at^d,at^{d^2},...,at^{d^m},t)$. This is a closed embedding with some image $C_a=\theta_a(\A^1)$. In particular, for $a=0$ $\theta_0: \A^1\to
\A^{m+1}$ is the morphism sending $t$ to $(0,0,...,0,t)$.
This is a closed embedding with image $C_0=\theta_0(\A^1)$.  This is the last coordinate axis $\A^1_{m+1}$ в $\A^{m+1}$.

By Note~\ref{r:quasi-finite}, the morphism

 \[\psi=(\pi, \phi,f_1 \circ pr_2 \circ g,\dots,f_k \circ pr_2 \circ g)|_{g^{-1}(S' \times X)}\]

is quasifinite. For a point $b=pr_1(v)\in B$, there is a polynomial $F \in k[x_1,\dots,x_{m+1}]$ from the $c$-defining set, such that $F(-,b)$ is nonzero and its zeros $Z(F(-,b))$ in $\A^{m+1}_b$
contain $(\phi,f_k \circ pr_2 \circ g)(W(u)).$ The latter condition is equivalent to requiring that for the polynomial $\widetilde{F} \in k[x_1,\dots x_{m+k}],$ defined as 

\[\widetilde{F}(x_1,\dots x_{m+k})=F(x_1, \dots,x_m,x_{m+k}),\]

\[Z(\widetilde{F}) \supseteq \psi \big ( W(b) \big ).\]

 For a given $0\neq a\in k(v)$ the scheme
$Y_a(b)$ is contained in the scheme

   \[\psi^{-1}(Z(F(-,b))\cap C_a).\]
	
The scheme $Z(F(-,b))\cap C_a)$ is isomorphic to the zero scheme of the polynomial $F(at^d,at^{d^2},...,at^{d^m},t)$ on the line $\A^1$ with coordinate $t$.  Therefore by~\cite[Lemma 7.1]{GNP} the set$Z(F(-,b))\cap C_a)$ is finite in this case, hence so is $Y_a(b)$.

For $a=0$ the set $Y_0(b)$ coincides with the closed subset
$Z(\phi_1,...,\phi_m,f_1,\ldots,f_{k-1}) \cap g^{-1}(S' \times X)$ of $W$. That is quasifinite over $U$,
because $c=(Z,W,\phi;g) \in (Fr_m^{qf,k-1})^{<d,k}(B,\pair M,\pair B;f_1, \cdots f_n)$, proving the second statement.
\end{proof}

\begin{proposition}\label{ZF_qf_and_ZF}
For any integers $m,k\geq 0,$ the morphism

\[C_*In^k: C_*\ZF_m^{qf,k-1}(-,\pair M,\pair B;f_1, \cdots f_n)  \to C_*\ZF_m^{qf,k}(-,\pair M,\pair B;f_1, \cdots f_n)\]

of complexes of presheaves of abelian groups is a schemewise quasiisomorphism on the category $AffSm/k$.
\end{proposition}

\begin{proof}
After applying the Suslin complex $C_*$ to Corollary~\ref{cor:exhausting_2} and taking the $l$th cohomology we get a colimit diagram

\[\begin{tikzcd}
\cdots \arrow[d] & \cdots \arrow[d] \\
 H^l \l( (C_*\ZF_m^{qf,k-1})^{<d-1,k}(-,\pair M,\pair B;f_1, \cdots f_n) \arrow[d] \arrow[ddd, start anchor = west,end anchor = west,bend right] \r) \arrow[r] & H^l\l( (C_*\ZF_m^{qf,k})^{<d-1}(-,\pair M,\pair B;f_1, \cdots f_n) \r) \arrow[d] \arrow[ddd,start anchor = east, end anchor = east, bend left]  \\
 H^l \l((C_*\ZF_m^{qf,k-1})^{<d,k}(-,\pair M,\pair B;f_1, \cdots f_n) \arrow[d] \arrow[dd, start anchor = west,end anchor = west, bend right] \r)\arrow[r] & H^l \l((C_*\ZF_m^{qf,k})^{<d}(-,\pair M,\pair B;f_1, \cdots f_n) \r)\arrow[d] \arrow[dd, start anchor = east, end anchor = east, bend left]  \\
 \cdots & \cdots  \\
H^l \l(C_*\ZF_m^{qf,k-1}(-,\pair M,\pair B;f_1, \cdots f_n) \r) \arrow[r] & H^l \l( C_*\ZF_m^{qf,k}(-,\pair M,\pair B;f_1, \cdots f_n)\r) \\
\end{tikzcd}.\]

also from Lemma~\ref{l:homotopy_properties} and homotopy invariance of the cohomology of the Suslin complex, we have a commutative diagram

\[\begin{tikzcd}
H^l \l( (C_*\ZF_m^{qf,k-1})^{<d,k}(-,\pair M,\pair B;f_1, \cdots f_n) \r) \arrow[r,"H^l(C_*In_d^k)"]  \arrow[d,"H^l(C_*J_d^k)"] & H^l \l( (C_*\ZF_m^{qf,k})^{<d}(-,\pair M,\pair B;f_1, \cdots f_n) \r) \arrow[d,"H^l(C_*I_d^k)"]  \arrow[dl,"H^l(C_*T_d^k)"']  \\
      H^l \l(  C_*\ZF_m^{qf,k-1}(-,\pair M,\pair B;f_1, \cdots f_n) \r) \arrow[r,"H^l(C_*In^k)"]  & H^l \l( C_*\ZF_m^{qf,k}(-,\pair M,\pair B;f_1, \cdots f_n) \r)     
\end{tikzcd}.\]	
Since morphisms of presheaves are determined by the images of specific sections, it follows that $H^l(C_*In^k)$ is an isomorphism Indeed, each class

\[c \in H^l \l( C_*\ZF_m^{qf,k}(B,\pair M,\pair B;f_1, \cdots f_n) \r)\]

 lies in the image of some $H^l(C_*I_d^k).$ But then, by vommutativity of the diagram, it is also in the image of $H^l(C_*In^k).$ For injectivity, let 

\[c \in H^l \l(  C_*\ZF_m^{qf,k-1}(B,\pair M,\pair B;f_1, \cdots f_n) \r),\]

 and $H^l(C_*In^k)(c) = 0$. There exists 

\[\tilde{c} \in H^l \l( (C_*\ZF_m^{qf,k-1})^{<d,k}(B,\pair M,\pair B;f_1, \cdots f_n) \r),\]

 such that $c=  H^l(C_*J_d^k) (\tilde{c}).$ Since 

\[H^l(C_*I_d^k) \circ H^l(C_*In_d^k) (\tilde{c})=0,\]

 then, after passing to a sufficiently large $d,$ we may suppose that $ H^l(C_*In_d^k) (\tilde{c})=0.$ But then 

\[c= H^l(C_*T_d^k) \circ H^l(C_*In_d^k) (\tilde{c})=0.\]

Since this is true for any $l$, the argument above shows that $C_*In^k$ is a quasiisomorphism.
\end{proof}

Now we are in a position to prove Theorem~\ref{th:moving}.
\begin{proof}
We have a chain

\[\begin{tikzcd}
 C_* \ZF^{qf}(B,\pair M,\pair B) \arrow[d,equals] \arrow[ddddd, start anchor = west,end anchor = west,bend right =50,"C_* \bb Z i"']\\
C_*\ZF_m^{qf,0}(B,\pair M,\pair B;f_1, \cdots f_n) \arrow[d,"C_*In^1"] \\
C_*\ZF_m^{qf,1}(B,\pair M,\pair B;f_1, \cdots f_n) \arrow[d,"C_*In^2"]\\
\cdots \arrow[d,"C_*In^n"]\\
C_*\ZF_m^{qf,n}(B,\pair M,\pair B;f_1, \cdots f_n) \arrow[d,equals]\\
C_* \ZF(B,\pair M \wedge \pair B)
\end{tikzcd}\]

All of the morphisms $C_*In^k$ in it are quasiisomorphisms by Proposition~\ref{ZF_qf_and_ZF}, hence their composition $C_* \bb Z i$ is also a quasiisomorphism.

\end{proof}

In what follows we will need another variant of the moving lemma. To state it we inductively define another object.

\begin{definition} \label{multi-d}
For $d_1, \ldots d_k \in \bb N,$ $(Fr_m^{qf,k})^{<d_1, \ldots, <d_k}(B,\pair M,\pair B;f_1, \cdots f_n)$ is the subset in $(Fr_m^{qf,k})^{<d_k}(B,\pair M,\pair B;f_1, \cdots f_n)$ consisting of classes containing such correspondences $c,$ for which $t_{d_k}(c) \in (Fr_m^{qf,k-1})^{<d_1, \ldots, <d_{k-1}}(B,\pair M,\pair B;f_1, \cdots f_n).$
\end{definition}

For $k>1,$ $(Fr_m^{qf,k})^{<d_1, \ldots, <d_k}(B,\pair M,\pair B;f_1, \cdots f_n)$ may not embed into each other, but their union is the set $Fr_m^{qf,k}(B,\pair M,\pair B;f_1, \cdots f_n).$

~\\
If $c \in (Fr_m^{qf,k})^{<d_1, \ldots, <d_k}(B,\pair M,\pair B;f_1, \cdots f_n),$ then $t_{d_1} \circ \cdots \circ t_{d_k}(c)$ gives an element in $Fr_m^{qf,0}(B,\pair M,\pair B;f_1, \cdots f_n),$ the image of which in $Fr_m^{qf,k}(B,\pair M,\pair B;f_1, \cdots f_n)$ is connected to $c$ by a sequence of $k$ homotopies.

In particular, for $k=n$ we get the following.

\begin{lemma}\label{l:homotopy_chain_to_qf}
For each $c \in Fr_m(B,\pair M \wedge \pair B),$ there exists a $\tilde{c} \in Fr_m^{qf}(B,\pair M,\pair B),$ such that $i(\tilde{c})$ is connected to $c$ by a sequence of homotopies.
\end{lemma}

We will need the following variation of the statement that these subset exhaust everything:

\begin{lemma} \label{l:d_for_finite_n_of_c}
For any finite set of correspondences $c_i$ from affines $A_i$ to $(X,U),$ there is a sequence $d_1, \cdots,d_n$, such that all the $c_i$ lie in $(Fr_m^{qf,n})^{<d_1, \ldots, <d_n}(B,\pair M,\pair B;f_1, \cdots f_n)$
\end{lemma}

\begin{proof}
It is enough to choose $d_i$ for the correspondence $\coprod{c_i}:\coprod{A_i} \to (X,U).$
\end{proof}

\section{Local connectedness of the spaces $C_*Fr(-,(X,U))$}\label{s:connectedness}

Since the spaces of the type $C_*Fr(-,(X,U))$ are the $0$th spaces of spectra of our interest, we first prove that they are locally connected, which will allow us to show that the corresponding spectra are $\Omega$-spectra. Then we show that under appropriate conditions these spaces have higher connectedness, which implies the same connectedness for the corresponding spectra.

\begin{lemma}\label{l:connectedness}
For each $X\in Sm/k$ and each open $U \subseteq X,$ if $U$ intersects each connected component of $X$ in a nonempty subscheme, then the pointed simplicial presheaf 
$C_*\Fr(-,(X,U)$ is Nisnevich-locally connected on $Sm/k$
\end{lemma}

We will need the following lemmas. 

\begin{lemma}\label{l:separable}
(see~\cite[Corollary 16.17(b)]{E})
Let $k$ be a perfect field. Then any finitely generated extension of $k$ can be written as a purely transcendental extension followed by a finite separable extension.
\end{lemma}

\begin{lemma}\label{l:incl_field}
Let $R$ be an equicharacteristic local henselian ring with a residue field finitely generated over a perfect field $k$. Then the canonical epimorphism $R \twoheadrightarrow R/m$ splits.
\end{lemma}

\begin{proof}
By definition, the field $k$ is embedded into $R.$ By the previous Lemma~\ref{l:separable}, $R/m$ can be written as a finite separable extension of some purely transcendental extension $k(\{x_i\}).$ We can find a family $\{\tilde{x_i}\}$ of preimages of this transcendental basis. By virtue of the ring being local, this choice gives an embedding $k(\{\tilde{x_i}\}) \hookrightarrow R.$ Next, using the separability of the remaining extension and the definition of a henselian ring, we extend this embedding to the whole of the field $R/m.$ 
\end{proof}

We also use the following lemma in order to avoid constantly shrinking the source of framed correspondences when proving local properties.

\begin{lemma}\label{l:Fr-colim}
(cf.~\cite[Lemma 2.23]{GPHoInv} )

Let $Y$ be an affine $k$-smooth variety, let $y \in Y$ be a point (not necessarily closed), and let $(X,U) \in SmOp(Fr_0(k))$ be a pair. Let $J$ enumerate the family of all possible affine \'etale neghbourhoods $Y_i$ of the point $y$ in $Y$ ($i \in J$), up to isomorphism. 

Then

\[
Fr_m(Y^h_y,(X,U)) = \colim \limits_{i \in J} Fr_m(Y_i,(X,U)). 
\]

The left-hand side of the equality is understood as the same geometric data as normally used for smooth schemes (see Definition~\ref{d:Fr-pair}, note the generality)
 
\end{lemma}

\begin{proof}[Proof of Lemma~\ref{l:Fr-colim}]
We first prove injectivity. Denote the canonical morphism $Y^h_y \to Y_i$ by $\hat{e}_i,$ and the corresponding inverse image morphism to $Y^h_y$ from $Y_i$ on framed correspondences by $\hat{e}_i^*.$

Let $i \in I,$ $c,c' \in Fr_m(Y_i,(X,U))$ be such that $\hat{e}_i(c_1)=\hat{e}_i(c_2).$ Let $c=(Z,W,\phi,g), c'=(Z',W',\phi',g').$ Consider the inverse images $(Z_y,W_y,\phi_y,g_y) $ and $(Z'_y,W'_y,\phi'_y,g'_y)$ of these correspondences to the local scheme $Y_{i,y}.$ Since $Y^h_y$ is faithfully flat over $Y_{i,y},$ and $\hat{e}_i(c_1)=\hat{e}_i(c_2),$ we have that $Z_y=Z'_y.$ If the ideals defining $Z$ and $Z'$ in $\A^m_{Y_i}$ are $I$ and $I'$, then $((I+I')/I)_y = 0,$ and $((I+I')/I')_y = 0.$ Since $Y$ is Noetherian, these equalities are already true after passing to some Zariski neighbourhood of the point $y$ in $Y.$ This, increasing $i$ to $i_1,$ we can assume that $Z_{i_1}=Z'_{i_1}.$  

Passing to a fibered product of $W_{i_1}$ and $W'_{i_1},$ we may consider them equal. Increasing $i_1$ to $i_2$, we can assume that $W_{i_2}$ has no connected components with an empty stalk at the point $y.$ Then the map $Hom(W_{i_2},\A^m \times X) \to Hom(W_{i_2} \times_{Y_i}Y^h_y,\A^m \times X)$ is injective. Thus, it follows from $\hat{e}_i(c_1)=\hat{e}_i(c_2)$ that $\phi_{i_2}=\phi'_{i_2},g_{i_2}=g'_{i_2}.$ The two correspondences $c$ and $c'$ have become equal after restricting to $Y_{i_2},$ hence they have the same image in the colimit (the right-hand side of the desired equality).

We now show surjectivity. Let $c=(Z,W,\phi,g) \in Fr_m(Y^h_y,(X,U)).$ Let $I=(a_1, \cdots,a_k)$ be the ideal defining $Z$ in $\A^m_{Y^h_y}.$ It is finitely generated, hence its generators are all the inverse images of elements of one $k[\A^m_{Y_i}].$ Suppose that those elements generate the ideal $I_i$ there. 

Again because $Y^h_y$ is faithfully flat over $Y_{i,y},$ we have that $V(I_i)_y$ is finite over $Y_{i,y}.$ Then, after passing to some Zariski neighbourhood (increasing $i$ to $i_2$) we can assume that $Z_{i_2}=V(I_{i_2})$ is finite over $Y_{i_2}$. Indeed, it is enough to choose such a neighbourhood over which all of the finite set of generators of $k[\A^m_{Y_{i_2}}]/I_{i_2}$ as a $k[Y_{i_2}]$-algebra are integral elements. For this it is sufficient to take the neighbourhood $Y_{i_1}$ over which all of the integral dependence relations (over $Y_{i,y}$) are defined. We get polynomials with leading coefficient $1$ over $Y_{i_1}.$ Evaluating them at the generators gives elements of $k[\A^m \times Y_{i_1}]$ going to $0$ in the stalk $k[\A^m \times Y_y]$. Hence after another increase of $i_1$ to $i_2$ we can suppose them zero. Thus, we have a finite $Z_{i_2}.$

Next we find some $W_{i_3}$ of finite type such that $W_{i_3} \times_{Y_{i_3}} Y^h_y = W.$ (we can arbitrarily extend the generators and relations giving $W$ to some neighbourhood). The section $Z_{i_3} \to W_{i_3}$ is defined on the stalk, thus it can be extended to some $Y_{i_4},$ after increasing $i_3$ to $i_4$. Again by virtue of strict henselisation being faithfully flat, $W_{i_4,y}$ is etale over $\A^m \times Y_{i_4,y}.$ Since the property of a morphism being \'etale is local in the source, the morphism $W_{i_4} \to \A^m \times Y_{i_4}$ is \'etale after restricting to some $W'_{i_4}, W'_{i_4} \supset Z_y.$ Passing from $Y_{i_4}$ to $Y_{i_5}=Y_{i_4}-\pi(Z_{i_4}-W'_{i_4}),$  we get an \'etale neighbourhood $W'_{i_5}$ of the closed subset $Z_{i_5}$ in $\A^m \times Y_{i_5}.$  

Next, the morphism $(\phi_{i_6},g_{i_6}):W_{i_6} \to \A^m \times X$ can be taken as an arbitrary extension of the morphism $(\phi,g)$ defined on $W$. That is defined on some smaller neighbourhood $Y_{i_6}$. The subset $\{\phi_{i_6}=0\} \cap g_{i_6}^{-1}(S)$ may not be equal to $Z_{i_6},$ but (again by faithful flatness) their stalks in $y$ are equal, therefore after increasing $i_6$ to $i_7$ they are equal. Thus we have a correspondence $c_{i_7}=(Z_{i_7},W_{i_7},\phi_{i_7},g_{i_7}) \in Fr_m(Y_{i_7},(X,U)),$ such that $\hat{e}_{i_7}(c_{i_7})=c.$
 
\end{proof}

We give one proof of Lemma~\ref{l:connectedness} in the general case, and another one for quasiprojective $X$.

Proof 1 is more elementary and provides a rather explicit construction of the final homotopy. It also works in the general case. Futhermore, it is akin to the proof of the Proposition~\ref{p:connectedness} below, which might make comprehending these proofs together aneasier task for the reader.

Proof 2 is more concise and refers to other results. As given here, it requires $X$ to be quasiprojective, although it can be modified to cover the general case. Its essence is a corollary of the theory of almost elementary fibrations (a generalisation of Artin neighbourhoods) and nice triples; which is developed in the paper~\cite{PSV}. For a technically sophisticated reader this proof might be the easier one of the two. The case of a quasiprojective $X$ covers many interesting examples, and may be sufficient for many readers.

\begin{proof}[Proof 1]
Let $B$ be a regular local Henselian scheme $c=(Z,W,\phi,g) \in Fr_m(B,(X,U)).$ We show that $c$ is connected by a sequence of homotopies to the zero framed correspondence $\emptyset.$

By Lemma~\ref{l:homotopy_chain_to_qf} applied to the case $Q=(pt,\emptyset), P=(X,U),$ we can assume that $c \in Fr_m^{qf}(B,(pt,\emptyset),(X,U).$ This implies that $\phi^{-1}(0)$ is quasifinite over $B.$

Since $B$ is a local henselian scheme, by~\cite[I.4.2]{Mil}, $Y=\phi^{-1}(0)$ splits into a disjoint union $Y=Y_0 \coprod Y_1 \coprod \cdots \coprod Y_r,$ where the fiber $Y_0$ over the closed point is empty, and $Y_i$ are finite and connected. Since $Z$ is finite over $B,$ all of its components have a nonempty fiber over the closed point, and therefore are not contained in $Y_0.$ Refining $W$ by subtracting the closed subscheme $Y_0,$ we may consider $\phi^{-1}(0)$ finite over $B.$

We now construct ahomotopy for one connected component $Z_i.$ The complete homotopy will be assembled from these homotopies. 

Suppose $Y_i$ is the component containing $Z_i.$ We subtract from $W$ all the other $Y_j,$ and get the neighbourhood $W_i.$ 

Let $z$ be the closed point of $Z_i,$ $x=g(z).$ Since $Z_i$ is a local henselian scheme, using Lemma~\ref{l:incl_field}, we can assume that a morphism to $x$ is defined on $Z_i.$ Similarly, after refining $W_i,$ we may assume that it is defined on $W_i$ as a morphism $lift:W_i \to x.$ Thus the framed correspondence $c$ comes from some framed correspondence $\tilde{c} \in Fr_m(B,(X_{k(x)},U_{k(x)})) = Fr_m(B,X_{k(x)}/U_{k(x)}).$ At the same time, the point $x$ of the $k(x)$-variety $X_{k(x)}$ is ratonal. Let $d=dim_x(X),$ $a_1, \ldots a_d$ be local parameters in the point $x$ on $X_{k(x)}.$ It gives an \'etale morphism $e:\overset{\circ}X \to \A^d_{k(x)}$ from some neighbourhood  $\overset{\circ}X$ of the point $x.$ Refining $\overset{\circ}X,$ we may consider $\overset{\circ}X$ a Nisnevich neighbourhood of the point $0 \in \A^d_{k(x)}.$ Since $U_{k(x)},$ by the conditions, intersects $\overset{\circ}X,$ $dim(S_{k(x)} \cap U_{k(x)}) < d.$ We denote $\overset{\circ}S=S \cap \overset{\circ}X.$
It follows from Chevalley's theorem that the set-theoretic image $\overset{\circ}S$ в $\A^d_{k(x)}$ is contained in a proper closed subset $L \subset \A^d_{k(x)}.$ 

 The following argument, making up a step of the Noether normalisation lemma for an infinite field (with linear projections), shows that after change of basis we may assume that the projection $p_{1,\cdots,d-1}: \A^d_{k(x)} \to \A^{d-1}_{k(x)}$ onto the first $d-1$ coordinates is finite in restriction to $L,$ which, in turn, means that the restriction of the morphism $p_{1,\cdots,d-1} \times id :\A^d_{k(x)} \times \A^1 \to \A^{d-1}_{k(x)} \times \A^1$ to the subscheme $L \times \A^1)$ is finite.

\begin{lemma}
Let $k$ be an infinite field. Let $L \subsetneq \A^d_k$ be a proper closed subset. Then after a linear change of basis the projection $p_{1,\cdots,d-1}: \A^d_{k} \to \A^{d-1}_{k}$  onto the first $d-1$ coordinates is finite in restriction to $L.$
\end{lemma}

\begin{proof}
Let $F$ be a nonzero polynomial in $I(L)$ of degree $N$, $F_N$ be its leading homogeneous component. Since it is nonzero and the field is infinite, there is a point $a=[a_1,\dots,a_d] \in \bb{P}^{d-1}(k),$ such that $F_N(a) \neq 0.$ Up to reordering the coordinates and rescaling, we may assume that $a_1=1.$ Then if $\widetilde{F} (x_1,x_2, \dots,x_d)=F(x_1,x_2+a_2\cdot x_1, \dots,x_d+a_d \cdot x_1),$ we have $\widetilde{F}_N(1,0 \dots,0) \in k^*,$ and the leading coefficient of the polynomial $\widetilde{F}$ as a polynoial in $x_1$ lies in $k^*.$ Therefore, $x_1$ is intergral over $k[x_2+a_2 \cdot x_1, \dots, x_d + a_d \cdot x_1],$ and under an appropriate coordinate change the projection from $V(F),$ and, a fortiori, from $L$ turns out to be finite.   
\end{proof}

Define the morphism $\tilde{g}_A:W_i \times \A^1 \to \A^d_{k(x)}=\A^d \times x$ as $ e \circ g + ((0,0,\cdots,0,s),lift \circ pr_1),$ where $s$ is the homotopy coordinate. (refine $W_i$ in advance, so that $g(W_i) \subseteq \overset{\circ}X.$)

Let $Z'_i=\l(\tilde{g}_A|_{Y_i \times \A^1} \r) ^{-1} (L) \subset Y_i \times \A^1.$  Then  $Z'_i \subset Y_i \times L \times A^1$ is given by the equation $e \circ (g,lift)|_{Y_i} (y) + (0,0,\ldots,0,s) = l$ on triples $(y,l,s).$ Therefore  $Z'_i=\l(Y_i  \times_{\A^d} (L \times \A^1)\r)^{red}.$ Here $L \times \A^1 \to \A^d_{k(x)}=\A^d \times k(x)$ is given by $in \circ pr_1 -((0,0,\ldots,0,s),can \circ pr_1),$ where $can$ denotes the canonical morphism $L \to x.$  But that morphism coincides with the restriction of the morphism $p_{1,\cdots,d-1} \times id :\A^d_{k(x)} \times \A^1 \to \A^{d-1}_{k(x)} \times \A^1$ onto the subscheme $L \times \A^1),$ therefore it is finite by choice of basis above. Therefore, $Z'_i$ is finite over $Y_i,$ hence over $B.$ Note also that it contains $Z_i \times 0$

Since $Y_i$ is a local henselian scheme, $Z'_i$ consists of several connected component, each of them being a local henselian scheme. Choose among them the one containing $Z_i \times 0,$ and denote it by $Z^0_i$ Denoting by $W^0_i \subseteq W \times \A^1$ the complement to the rest of the components of $Z'_i,$ we have a closed embedding $Z^0_i \subseteq W^0_i.$ Since $e$ is \'etale, and the point $0 \in \A^d$ has a distinguished lifting $x \in \overset{\circ}X,$ some neighbourhood $\widetilde{W}_i$ of the point $(z,0) \in W^0_i$ has the property that the morphism $\tilde{g}_A,$ when restricted to it, lifts to a morphism $\tilde{g}:\widetilde{W}_i \to \overset{\circ}X$ sending $(z,0)$ to $x.$ Since $Z^0_i$ is local and henselian, it lifts into $\widetilde{W}_i,$ turning into its closed subscheme (after removing from $\widetilde{W}$ the extra components of the preimage, we may assume that this lifting is the whole of the preimage of $Z^0$). We will denote the restriction of the functions $\phi$ to any neighbourhood by the same letter, since at this stage we do not change them in any substantial way. Finally, we denote $\l(\tilde{g}|_{Z^0_i}\r)^{-1} (\overset{\circ}S)$ by $\widetilde{Z}_i.$ 

We are now under the following conditions:
\begin{itemize}
\item[$\bullet$]{A closed point $(z,0) \in \A^m \times B \times \A^1$}
\item[$\bullet$]{An \'etale neighbourhood $\widetilde{W}_i$ of the point $(z,0)$ in $\A^m \times B \times \A^1$ }
\item[$\bullet$]{A reduced closed subscheme $\widetilde{Z}_i \subseteq \widetilde{W}_i$ containing the point $(z,0),$ which is a local henselian scheme and is finite over$B$}
\item[$\bullet$]{A morphism $\tilde{g}:\widetilde{W}_i \to \overset{\circ}X$ and functions $\phi:\widetilde{W}_i \to \A^m,$ such that $\phi^{-1}(0)\cap \tilde{g}^{-1}(S^0) = \widetilde{Z}_i.$ }
\end{itemize}

By~\cite[Lemma 4.2]{GNP} applied to the situation when as the tuple $(V,Z,U,Y)$ we take $(\A^m \times B \times \A^1, (z,0) ,B,\widetilde{Z}_i),$ $\widetilde{Z}_i$ is identified with a closed subscheme in $\A^m \times B \times \A^1.$ Thus these data define a level $m$ explicit framed correspondence from $B \times \A^1$ to $(\overset{\circ}X,\overset{\circ}X - \overset{\circ}S).$ From the definitions of $\overset{\circ}X$ and $\overset{\circ}S$ it is clear that the same data, composed with the embedding $\overset{\circ}X \hookrightarrow X_{k(x)},$ and the natural projection $X_{k(x)} \to X$ define a correspondence $h_i \in Fr_m(B \times A^1, (X,X-S)).$

If $c_i$ is the component of the correspondence $c$ with support $Z_i,$ we prove that $h_i$ is a homotopy between $c_i$ и $\emptyset.$ Indeed, $i_0^* h_i$ is given by the data 

\[(\widehat{Z}_i^0=\widetilde{Z}_i \times_{W \times \A^1} (W \times 0), \widehat{W}_i^0=\widetilde{W}_i \times_{W \times \A^1} (W \times 0), \phi|_{\widehat{W}_i^0},\tilde{g}|{\widehat{W}_i^0}).\]

Note that the canonical morphism $\widetilde{W}_i \to W \times \A^1$ equips $\widehat{W}_i^0$ with an \'etale morphism $can$ to $W.$ The final data $\phi$ is derived from the data for the framed correspondence $c_i$ by composing with $can.$ The same is true about the data $g$ after composing with an \'etale morphism $e$ (defined in some neighbourhood)

\[
e \circ \tilde{g}|_{\widehat{W}_i^0} = e \circ g \circ can
\]

After passing to a smaller neighbourhood $\widehat{W}_i^0$ (by removing some connected components), we may assume that $\tilde{g}|_{\widehat{W}_i^0}) = g \circ can.$ In order to check that $i_0^* h_i=c_i,$ it remains to check that $\widehat{W}_i^0$ is an \'etale neighbourhood of $Z_i$ in $W.$ This is true because it is an \'etale neighbourhood of the point $z,$ and because $Z_i$ is a local henselian scheme. 

Since the support $\widetilde{Z}_i$ is a local henselian scheme, and its closed point lies over $0,$ the fiber $\widetilde{Z}_i \times_{W \times \A^1} (W \times 1)$ is empty, therefore $i_1^* h_i = \emptyset.$

Passing to the disjoint union over $i,$
 
\[h=\l(\coprod \widetilde{Z}_i,\coprod \widetilde{W'}_i, \coprod \phi, \coprod \tilde{g}\r) \in Fr_m(B \times \A^1, (X,U)).\]

We have $(i_0)^* h = c, (i_1)^* h = \emptyset.$ The homotopy $h$ we have constructed connects the correspondence $c$ with the empty correspondence. 
\end{proof}

\begin{proof}[Proof 2, in the quasiprojective case.]

Suppose $X$ quasiprojective.

Let $B$ be a regular local Henselian scheme, $c=(Z,W,\phi,g) \in Fr_m(B,(X,U)).$ We prove that $c$ is connected by a sequence of homotopies to the zero framed correspodence $\emptyset.$

By Lemma~\ref{l:homotopy_chain_to_qf} applied to the case $Q=(pt,\emptyset), P=(X,U),$ we may assume that $c \in Fr_m^{qf}(B,(pt,\emptyset),(X,U).$ This means that $\phi^{-1}(0)$ is quasifinite over $B.$

Since $B$ is a local henselian scheme, by~\cite[I.4.2]{Mil}, $Y=\phi^{-1}(0)$ splits into a disjoint union $Y=Y_0 \coprod Y_1 \coprod \cdots \coprod Y_r,$ where te fiber of $Y_0$ over the closed point is empty, and the $Y_i$ are finite and connected. Since $Z$ is finite over $B,$ all of its components have a nonempty fiber over the closed point, therefort they are not contained in $Y_0.$ Refining $W$ by removing the closed subscheme $Y_0,$ we may assume that $\phi^{-1}(0)$ is finite over $B.$

Now (after refining $W$) the tuple $(Y,W,\phi,g)$ defines a correspondence $\tilde{c} \in Fr_m(B,X) =Fr_m(B,(X,\emptyset)).$ Let $x_1, \cdots, x_l$ be all the images of the closed points of $Y$ under $g.$ (these are points in $X$, not necessarily closed.) Since $X$ is quasiprojective, all of them are contained in some addine open subset of $X.$ Let $V=Spec(\cc O_{X,\{x_1, \dots, x_l\}})$ be the semilocal scheme playing the role of an infinitesimal Zariski neighbourhood of the points $x_1, \dots, x_l.$ We prove that the correspondence $\tilde{c}$ in some sense factors through $V.$

Followng the thought of~\cite[Proposition 3.8 and Definition 3.10]{GP}, we introduce for a $k$-smooth variety $B' \in Sm(k)$ and a pointed Nisnevich sheaf $\cc{F} \in Nis_*(Sm(k))$ the sets $\cc Fr _m(B',\cc{F})$ of level $m$ framed correspondences from $B'$ to $\cc{F}.$

\begin{definition}
Let $B' \in Sm(k),$ $\cc{F} \in Nis_*(Sm(k)).$  The set $\cc Fr _m(B',\cc{F})$ of level $n$ framed correspondences from $B'$ to $\cc{F}$ is the set $Hom_{Nis_*}(B_+ \wedge \bb{P}^{\wedge m}, \cc F \wedge T^{\wedge m}).$
\end{definition} 

\begin{note}\label{n:Fr_m-Sh}
If $F=X''/U''$ for $(X'',U'') \in SmOp(Fr_0(k)),$ there is a natural bijection $\Fr _m(B',(X'',U'')) \simeq \cc Fr_m (B',X''/U'').$ (See~\cite[Proposition 3.8]{GP}
\end{note}

\begin{note}\label{n:Fr-exact}
The functor $\cc Fr_m(B',-)$ is left exact, i.e. respects limits. Indeed, the smash product $\wedge$ and $Hom_{Nis_*}(B' \wedge \bb{P}^{\wedge m},-)$ have that property, hence so does their composition.  
\end{note}

Note that $V$ is the limit of affine neighbourhoods $V^j$ of points $x_1, \cdots, x_l$ (here we use quasiprojectivity of $X$). Also, since the support $Y$ of the correspondence $\tilde{c}$ is semilocal, after passing to a smaller $W$ we get that the correspondence $\tilde{c}$ can uniquely be factored through any neighbourhood $V^j.$ Hence we get a compatible system of elements, i.e. an element of the inverse limit. By Note~\ref{n:Fr-exact}, this means that the correspndence $\tilde{c} \in \Fr_m(B,X)= \cc Fr (B, Hom(-,X_+))$ factors into the composition of $\tilde{\tilde{c}} \in \cc Fr_m(B,Hom(-,V_+))$ and the canonical embedding $in: Hom(-,V_+) \to Hom(-,X_+).$ 

 Meanwhile the original correspondence $c$ can be written as the composition of $\tilde{c}$ with the canonical factorisation morphism $p: (X,\emptyset) \to (X,U),$ which is the identity on $X.$ Translating to the language of pointed Nisnevich sheaves, $p$ corresponds to the canonical projection $X_+ \to X/U.$ Combining these two factorisations, we get $c= p \circ in \circ \tilde{\tilde{c}}.$ It suffices to show that the morphism $p \circ i: Hom(-,V_+) \to X/U$ is naively homotopical to the distinguished constant morphism $* \in Hom_{Nis_*}(Hom(-,V_+),X/U).$ The latter statement was proved for a quasiprojective $X$ in~\cite[Theorem 2.1]{PanMove}, using the technique of almost elementary fibrations and nice triples.

\end{proof}

\begin{corollary} \label{c:connectedness_simplicial}
For any simplicial object $P_{\bullet}$ in $SmOp(Fr_0(k)),$ if $P_0=(X_0,U_0),$ and $U_0$ intersects each connected component of $X_0,$ then the bisimplicial presheaf of pointed sets 
$C_*\Fr(-,P_{\bullet})$ is locally connected in the Nisnevich topology on $Sm/k.$
\end{corollary}

\begin{proof}
It suffices to prove local connectedness of the $0$th space, and that is the content of the previous Lemma~\ref{l:connectedness}.
\end{proof}

We prove a more general statement on higer connectedness:

\begin{proposition} \label{p:connectedness}
Let $r>0,$ and let $(X,U=X-S) \in SmOp(Fr_0(k))$ be such a pair that $codim_{X_i}(S \cap X_i)>r$ in each connected (Or, which is the same thing for a  $k$-smooth scheme, irreducible) component $X_i \subseteq X.$ Then the simplicial sheaf $C_*Fr(-,(X,U))$ is Nisnevich-locally $r$-connected.
\end{proposition}
\begin{plan}

The proof follows the plan below:
\begin{itemize}

\item[$\bullet$]{Since $C_*Fr(-,(X,U))$ is locally connected, by the stable Hurewicz theorem~\cite[Proposition II.6.30(i)]{Sch} if $H_k(M_{fr}((X,U))$ is locally zero for $k \leq r$, then $M_{fr}((X,U))$ is locally $r$-connected. since $M_{fr}((X,U))$ is locally an $\Omega$-spectrum, this is the same as $C_*Fr(-,(X,U))$ being locally $r$-connected. }

\item[$\bullet$]{Since the cohomology presheaves $H_k(M_{fr}((X,U))=\pi_k(\bb Z M_{fr}((X,U)))=\pi_k(LM_{fr}((X,U)))=H_k(C_* \ZF(-,(X,U))$ (The first equality follows from the assembly morphism being a stable equivalence, see~\cite[pp. 280-281]{Sch})  are homotopy invariant quasistable presheaves of abelian groups with $\ZF_*$-transfers, the map $H_k(M_{fr}((X,U))(B) \to H_k(M_{fr}((X,U))(Spec(k(B))) $ is injective, where $B$ is the henselisation of a local ring of a $k$-smooth variety}

\item[$\bullet$]{Using the homotopy Mayer-Vietoris triangle for pairs, the general case is reduced to the case of a "sufficiently small" $X.$}

\item[$\bullet$]{For a field $F/k,$ we equip the functor $Fr_m(-,(X,U))$ on the category $AffSm(F)$ with a family of subfunctors $Fr_m^{<d_1, \ldots <d_n,v}(-,(X,U);f_1,\ldots f_n),$ and any finite set of framed correspondences from affine smooth $F$-schemes of dimension no more than $r$ is included into one of such subfunctors}

\item[$\bullet$]{For each functor $Fr_m^{<d_1, \ldots <d_n,v}(-,(X,U);f_1,\ldots f_n)$ we define a chain of homotopies contracting it within $Fr(-,(X,U))$. Thus the simplicial set 
\[C_*Fr_m^{<d_1, \ldots <d_n,v}(Spec(F),(X,U);f_1,\ldots f_n) \subset C_*Fr_m(Spec(F),(X,U))\]
can be contracted inside $C_*Fr_m(-,(X,U)).$.}

\item[$\bullet$]{Any map from $S^k, k \leq r,$ to the geometric realisation $|C_*Fr_m(Spec(F),(X,U))|$ is homotopic to a map to the $r$-skeleton. Then, since the image of this map only intersects a finite number of simplices, and they are all of dimension $\leq r,$ it can be factored through one of the subcomplexes $C_*Fr_m^{<d_1, \ldots <d_n,v}(-,(X,U);f_1,\ldots f_n)$, which can be contracted inside $C_*Fr_m(-,(X,U)).$}
\end{itemize} 

\end{plan}

\begin{lemma}\label{l:connectedness_acyclicity}
Under the conditions of Proposition~\ref{p:connectedness}, the sheaf $C_*Fr(-,(X,U))$ is locally $r$-connected in the Nisnevich topology if and only if the spectrum $(M_{fr}((X,U))$ is locally $r$-acyclic.
\end{lemma}

\begin{proof}
By~\cite[Theorem 6.5]{GP}, $(M_{fr}((X,U))$ is locally an $\Omega$-spectrum. Therefort local $r$-connectedness of its zeroth space $C_*Fr(-,(X,U))$ is equivalent to the local $r$-connectedness of the entire spectrum $(M_{fr}((X,U)).$ By the stable Hurewich theorem~\cite[Proposition II.6.30(i)]{Sch}, the first nonzero homotopy group is isomorphic to the corresponding homology group. Hence local $r$-acyclicity is equivalent to local $r$-connectedness.
\end{proof}

\begin{lemma}\label{l:acyclicity_field}
Under the conditions of Proposition~\ref{p:connectedness}, if $B \in Sm(k), b \in B,$ then for any $k$ the map $H_k(M_{fr}((X,U))((\cc O_{B,b})^h_b) \to H_k(M_{fr}((X,U))(Spec(k((\cc O_{B,b})^h_b)))$ is injective.
\end{lemma}

\begin{proof}
$H_k(M_{fr}((X,U)))=\pi_k(\bb ZM_{fr}((X,U))).$

(see~\cite[Definition II.6.24, end of p. 280]{Sch})

by Theorem~\ref{ZM_fr_and_LM_fr},  $\pi_k(\bb ZM_{fr}((X,U)))=\pi_k(LM_{fr}((X,U))).$ The latteris a homotopy invariant quasistable presheaf with $\ZF_*$-transfers. The statement of the Lemma folows from~\cite[Theorem 3.15(3')]{GPHoInv} \footnote{In the precious version of the cited text that was~\cite[Theorem 2.15(3')]{GPHoInv}} \}.  
\end{proof}

\begin{lemma}\label{l:cover}
Suppose that the statement of the Theorem is true for all pairs $(X',U'),$ where $X'$ is quasiaffine and \'etale over some affine space $\A^d.$ Then it is true for an arbitrary pair $(X,U) \in SmOp(Fr_0(k)).$
\end{lemma}
\begin{proof}
Since $X$ is a $k$-smooth scheme, $X$ can be covered by open $X_i,$ which are quasiaffine and equipped with \'etale maps to $\A^{d_i}.$ Suppose the $X_i$ enumerated by ordinals, and prove via induction that the simplicial sheaf  $C_*Fr_m(-,(\bigcup \limits_{i \leq \alpha} X_i, U \cap \bigcup \limits_{i \leq \alpha} X_i ))$ is locally $r$-connected. We only prove the induction step $\alpha \to \alpha+1,$ since the base $\alpha=0$ and the transfinite step are obvious. By Proposition~\ref{p:HoMV}, denoting $\bigcup \limits_{i \leq \alpha} X_i$ by $\widetilde{X},$ we have a distinguished triangle

\[\begin{tikzcd}
M_{fr}\l(\l(X_{\alpha+1} \cap \widetilde{X},U \cap X_{\alpha+1} \cap \widetilde{X}\r)\r) \arrow[d] \\
M_{fr} \l(\l(X_{\alpha+1}, U \cap X_{\alpha+1}\r)\r) \vee M_{fr} ((\widetilde{X},U \cap \widetilde{X})) \arrow[d] \\
 M_{fr} \l(\l(X_{\alpha+1} \cup \widetilde{X},U \cap \l(X_{\alpha+1} \cup \widetilde{X}\r)\r)\r),
\end{tikzcd}\]

We already know, by Lemma~\ref{l:connectedness}, that the $0$th spaces of the spectra in this triangle are connected, and then, by~\cite[Theorem 6.5]{GP}, these spectra a locally $\Omega$-spectra. Thus, they have the same local connectedness as their $0$th spaces. But the triangle above gives rise to a homotopy long exact sequence: 

\[\begin{tikzcd}
\pi_i \l( M_{fr} \l(\l(X_{\alpha+1}, U \cap X_{\alpha+1}\r)\r) \vee M_{fr} ((\widetilde{X},U \cap \widetilde{X})) \r)  \arrow[d] \\
 \pi_i \l( M_{fr} \l(\l(X_{\alpha+1} \cup \widetilde{X},U \cap \l(X_{\alpha+1} \cup \widetilde{X}\r)\r)\r) \r) \arrow[d]\\
\pi_{i-1} \l( M_{fr}\l(\l(X_{\alpha+1} \cap \widetilde{X},U \cap X_{\alpha+1} \cap \widetilde{X}\r)\r) \r),\\
\end{tikzcd}\]

For each $i \leq r,$ since the outer terms of the presented part of the sequence are zero by the induction hypothesis, the middle term will also be zero, proving the induction step.

\end{proof}

In what follows we suppose that
\begin{conditions} \label{c:local}
 $X$ is quasiaffine, $S$ is defined by $n$ equations $f_1, \ldots f_n,$ and an \'etale morphism $e:X \to \A^d$ is given on $X;$ $L \subset \bb A^d$ is a closed subscheme in $\A^d$ of codimension $>r,$ through which $e|_S$ factors; $F$ is the field of rational functions of a scheme of the type $(\cc O_{B,b})^h_b$.
\end{conditions}

\begin{definition}
Under the conditions~\ref{c:local}, let $A$ be an affine smooth scheme over $Spec(F).$ Define the subset 

\[Fr_m^{<d_1, \ldots <d_n,v}(A,(X,U);f_1,\ldots f_n) \subseteq (Fr_m^{qf,n})^{<d_1, \ldots <d_n}(A,(pt,\emptyset),(X,U);f_1,\ldots f_n)\]

(see Definition~\ref{multi-d}) by the following condition:

Let $c \in Fr_m^{<d_1, \ldots <d_n}(A,(X,U);f_1,\ldots f_n).$ Consider a correspondence $(Z,W,\phi,g)=t_0 \circ \cdots \circ t_n (c).$ It has the $F$-variety of zeros $Y=\{\phi=0\}$ quasifinite over $A.$ A morphism of schemes (varieties over different fields) $e \circ g|_Y:Y \to \A^d=\A^d_k$ is defined.

Let $v \in F^d-0.$ Consider a linear homotopy $\tilde{g}_v=g+sv:W \times \A^1 \to \A^d.$ We say that $c \in Fr_m^{<d_1, \ldots <d_n,v}(A,(X,U);f_1,\ldots f_n),$ if $Z \times 0$ is a connected component of $\tilde{g}_v^{-1}(L) \cap (Y \times \A^1).$
\end{definition}

\begin{lemma} \label{l:almost_all_v}
Under the conditions~\ref{c:local}, for any $A \in AffSm(Spec(F)), dim(A) \leq r$, for any $c \in Fr_m^{<d_1, \ldots <d_n}(A,(X,U);f_1,\ldots f_n),$ $c \in Fr_m^{<d_1, \ldots <d_n,v}(A,(X,U);f_1,\ldots f_n)$ for almost all $v.$
\end{lemma}

\begin{proof}
Consider the morphsim $line:(Y \times_k L) - \Gamma(e \circ g|_Y) \to \bb P^{d-1}_F$ given by passing a line through two (different) given points. 
The $F$-variety map $line$ has source dimension not greater than $r+(d-r-1)=d-1,$ and target dimension $d-1,$ thus this morphism is quasifinite over the generic point, hence over some open subset. Let $v$ be such that the fiber over $[v] \in \bb P^{d-1}$ is finite. For the homotopy $\tilde{g}_v=g+sv:W \times \A^1 \to \A^d,$  $\tilde{g}_v^{-1}(L) \cap (Y \times \A^1),$ except for $Z \times 0,$ contains only a finite set $M$ of other points, since the other points give rise to points in the fiber of $line$ over $[v].$ Since $Z \times 0$ is closed, and a finite set of (closed) points $M$ is also closed, $Z \times 0$ is a connected component.
\end{proof}

\begin{lemma}\label{l:contractible}
Under the conditions~\ref{c:local}, for each functor $Fr_m^{<d_1, \ldots <d_n,v}(-,(X,U);f_1,\ldots f_n)$ there are $n+1$ natural homotopies

\[\tilde{h}_0, \dots, \tilde{h}_{n}:Fr_m^{<d_1, \ldots <d_n,v}(-,(X,U);f_1,\ldots f_n) \to Fr_m(- \times \A^1,(X,U)),\]

deforming the embedding $Fr_m^{<d_1, \ldots <d_n,v}(-,(X,U);f_1,\ldots f_n) \to Fr_m(-,(X,U))$ into the distinguished constant map $const_{\emptyset}.$
\end{lemma}

\begin{proof}
One should take $\tilde{h}_1, \ldots, \tilde{h}_n$ to be 

\[h_{d_1} \circ t_{d_2} \circ \cdots \circ t_{d_n}, h_{d_2} \circ t_{d_3} \circ \cdots \circ t_{d_n} , \ldots, h_{d_n}.\]

These homotopies deform any correspondence to a "quasifinite" one. The rest of the proof is devoted to constructing the final homotopy.

We have a correspondence $t_{d_1} \circ t_{d_2} \circ \cdots \circ t_{d_n}(c).$ It has $Y=\phi^{-1}(0)$ quasifinite over $A.$ Consider a linear homotopy $\tilde{g}_v=g+sv:W \times \A^1 \to \A^d.$ It has $\tilde{g}_v^{-1}(L) \cap (Y \times \A^1)$ have $Z \times 0$ as a connected component. $X \times_{A^d} (W \times \A^1-M)$ maps in an \'etale manner to $(W \times \A^1-M),$ and that map has a section on a closed subset $Z \times 0.$ After elimiating the extra components of the preimage of $Z \times 0,$ we have an \'etale neighbourhood $\widetilde{W}$ of the closed subvariety $Z \times 0 \subset \A^m \times A \times \A^1.$ Denote $\tilde{\phi}=\phi \circ pr_1.$ Denote by $\tilde{g}$ the morphism $pr_1:\widetilde{W} \to X.$ Then $\tilde{\phi}^{-1}(0) \cap \tilde{g}^{-1}(S) = Z \times 0.$ This support is finite over $A \times 0,$ and, a fortiori, over $A \times \A^1.$ Meanwhile, the fiber of the support over $1$ is empty, therefore the homotopy deforms $t_1 \circ \cdots \circ t_n(c)$ to the empty correspondence $\emptyset.$

It is easy to check that these constructions behave well with under morphisms between various affine smooth $F$-schemes $A,$ and therefore they give a natural transformation 

\[\tilde{h}_0: Fr_m^{<d_1, \ldots <d_n,v}(-,(X,U);f_1,\ldots f_n) \to Fr_m(- \times \A^1,(X,U)).\]

\end{proof}

\begin{lemma} \label{l:connected_over_field}
Under the conditions~\ref{c:local}, the simplicial set $C_*Fr(Spec(F),(X,U))$ is $r$-connected.
\end{lemma}

\begin{proof}
%Достаточно доказать, что для каждого конечного симплициального множества $Q_{\bullet},$ порожденного симплексами размерности $\leq r,$ любой морфизм $Q_{\bullet} \to C_*Fr(Spec(F),(X,U))$ можно стянуть в точку
Let $p:S^k \to |C_*Fr(Spec(F),(X,U))|$ be a continuous map of topological spaces, $k \leq r.$ Up to (pointed) homotopy, we may assume that $p(S^k)$ lies in the $r$-skeleton. By compactness, $p(S^k)$ only intersects a finite number of nondegenerate simplices. Let $m$ be the level at which all these simplices are represented by level $m$ framed correspondences. 

By Lemma~\ref{l:d_for_finite_n_of_c}, they all lie in some $(Fr_m^{qf,n})^{<d_1, \ldots <d_n}(-,(pt,\emptyset),(X,U);f_1,\ldots f_n).$ By Lemma~\ref{l:almost_all_v}, they all lie in $Fr_m^{<d_1, \ldots <d_n,v}(-,(X,U);f_1,\ldots f_n)$ for almost any $v.$ Choose such a $v.$ 

Now $p$ factors through the simplicial subset $C_*Fr_m^{<d_1, \ldots <d_n,v}(-,(X,U);f_1,\dots f_n).$ But this simplicial subset can be contracted within the larger simplicial set, by Lemma~\ref{l:contractible}. Thus the map $p$ is homotopical to the trivial one, and the simplicial set $C_*Fr(Spec(F),(X,U))$ is $r$-connected.
\end{proof}

We now move on to proving Proposition~\ref{p:connectedness}.

\begin{proof}
By Lemma~\ref{l:connectedness_acyclicity}, proving the Proposition is equivalent to showing local $r$-acyclicity of the spectrum $M_{fr}((X,U)),$ i.e. to showing that for a smooth $B,$ and $b \in B,$ $H_k(M_{fr}((X,U)))(Spec( (\cc O_{B,b})^h_b) )=0.$ 

By Lemma~\ref{l:acyclicity_field}, this group maps injectively into $H_k(M_{fr}((X,U)))(Spec(F)),$ where $F$ is the fraction field of the local henselian ring. Hence it suffices to show that $C_*Fr(Spec(F),(X,U))$ is locally $r$-connected.

By Lemma~\ref{l:cover}, it suffices to consider the case when $X$ is quasiaffine and \'etale over some $\A^d.$ We may thus suppose we are under the conditions~\ref{c:local}. 

By Lemma~\ref{l:connected_over_field}, in the case of the conditions~\ref{c:local}, $C_*Fr(Spec(F),(X,U))$ is locally $r$-connected.
\end{proof}

\appendix
	
\section{Mayer-Vietoris triangles}

We need the following statement concerning Mayer-Vietoris triangles:

\begin{proposition} \label{p:MV}
Let

\[\begin{tikzcd}
V-Z \arrow[d] \arrow[r,hook] & V \arrow[d]\\
X-Z \arrow[r,hook] & X\\
\end{tikzcd}\]
be a Nisnevich square in $Sm/k.$
\begin{enumerate}
\item{The natural morphisms

\[C_*\ZF (V - Z) \to  C_*\ZF (X - Z) \oplus C_*\ZF (V)_{Nis} \to C_*\ZF (X) ,\]

induced by the inclusions define a distinguished triangle in the category of complexes of Nisnevich sheaves of abelian groups.}
\item{Let $W \subset X$ be an open subscheme. The natural morphisms

\[\begin{tikzcd}
C_*\ZF ((V - Z,W \times_X V - Z \cap (W \times_X V))) \arrow[d] \\
 C_*\ZF ((X - Z,W-Z\cap W)) \oplus C_*\ZF ((V,W \times_X V)) \arrow[d] \\
 C_*\ZF ((X,W)) ,
\end{tikzcd}\]

induced by the inclusions of pairs define a distinguished triangle in the category of complexes of Nisnevich sheaves of abelian groups.}
\end{enumerate}
\end{proposition}

\begin{proof}
The Proposition follows from two lemmas~\ref {l:ZF-MV} and~\ref{l:ZF -> C_*ZF} below, considering Note~\ref{n:ZF-good}.
\end{proof}

\begin{lemma} \label{l:ZF-MV}
Under the conditions of Proposition~\ref{p:MV}
\begin{enumerate}
\item{The sequence of Nisnevich sheaves of abelian groups
\[0 \to \ZF (V - Z) \to  \ZF (X - Z) \oplus \ZF (V) \to \ZF (X) \to 0\]
induced by inclusions of schemes is exact.}
\item{Let $W \subset X$ be an open subscheme. The sequence of Nisnevich sheaves of abelian groups 

\[\begin{tikzcd}
0 \arrow[d] \\
\ZF ((V - Z,W \times_X V - Z \cap (W \times_X V))) \arrow[d] \\
 \ZF ((X - Z,W-Z\cap W)) \oplus \ZF ((V,W \times_X V)) \arrow[d]\\
\ZF ((X,W)) \arrow[d]\\
0
\end{tikzcd}\]

induced by inclusions of pairs is exact.}
\end{enumerate}
\end{lemma}

\begin{proof}
(1)

The sequence is derived from the pullback and pushout squares of sheaves of sets
\[\begin{tikzcd}
F_n (-,V - Z) \arrow[r] \arrow[d] & F_n (-,V) \arrow[d] \\
F_n (-,X - Z) \arrow[r] & F_n(-,X) 
\end{tikzcd}\]
The pushout property is reduced to the fact that over a local henselian scheme the support $Z$ of a framed correspondence is also a local henselian scheme, and therefore over it the Nisnevich square is trivial. Having a section over $Z$ allows us to lift the correspondence either to a correspondence to $V$, or to $X-Z.$

(2)

The sequence is derived from the pullback and pushout squares of sheaves of sets
\[\begin{tikzcd}
F_n (-,(V - Z,W \times_X V - Z \cap (W \times_X V))) \arrow[r] \arrow[d] & F_n (-,(V,W \times_X V)) \arrow[d] \\
F_n (-,(X - Z,W-Z\cap W)) \arrow[r] & F_n(-,(X,W)) 
\end{tikzcd}\]
The pushout property is reduced to the fact that over a local henselian scheme the support $Z$ of a framed correspondence is also a local henselian scheme, and therefore over it the Nisnevich square is trivial. Having a section over $Z$ allows us to lift the correspondence either to a correspondence to $(V,W \times_X V)$, or to $(X - Z,W-Z\cap W).$

\end{proof}

\begin{definition}
a presheaf $\cc{G}$ of abelian groups with $\ZF_*$-transfers is called \textbf{pre-quasistable}, if the presheaves $H^i(C_* \cc{G})$ are quasistable (and obviously homotopy invariant).
\end{definition}

\begin{note}
The category of pre-quasistable presheaves is closed under limits, colimits and extensions.
\end{note}

\begin{note} \label{n:ZF-good}
The sheaves of the type $\bb{Z}F ((X,U))$ are pre-quasistable. This is a consequence of the sheaves $F((X,U))$ being quasi-stable "up to naive $\A^1$-homotopy".
\end{note}

\begin{lemma} \label{l:ZF -> C_*ZF}
Let 

\[0 \to \cc{F}_1 \to \cc{F}_2 \to \cc{F}_3 \to 0\]

 be an exact sequence of pre-quasistable Nisnevich sheaves of abelian groups with $\ZF_*$-transfers. Then the morphisms

\[C_* \cc{F}_1 \to C_* \cc{F}_2 \to C_* \cc{F}_3 \]

define a distinguished triangle in the category of complexes of Nisnevich sheaves of abelian groups with transfers.
\end{lemma}

\begin{proof}
The exact sequence of sheaves gives a sequence of complexes of presheaves

\[0 \to C_* \cc{F}_1 \to C_* \cc{F}_2 \to C_* \cc{F}_3 \to 0,\]

 in which the cohomology are the presheaves of the kind $C_*(\cc G),$ where $\cc G$ is a quasistable presheaf with $(\cc G)_{Nis} = 0.$
By Lemma~\ref{l:G->C_*G} below, the presheaves $H^i(C_* \cc G)_{Nis} = 0.$ 

Therefore, the sequence of complexes of sheaves

\[0 \to C_* \cc{F}_1 \to C_* \cc{F}_2 \to C_* \cc{F}_3 \to 0\]

has acyclic cohomology. Thus this sequence of complexes defines a distinguished triangle in the derived category.
\end{proof}

\begin{lemma} \label{l:G->C_*G}
Let $\cc G$ be a pre-quasistable presheaf on the category $\ZF_*$, such that $\cc G_{Nis}=0.$ Then the complex of sheaves $(C_* \cc G)_{Nis}$ is acyclic.
\end{lemma}

\begin{proof}
We show that for each $i,$ $H^i(C_* \cc G)_{Nis}=0.$ We argue by induction on $i$. For $i<0$ the statement is obvious. Suppose that it is known for all $j<i.$ Then $\tau_i(C_* \cc G)_{Nis} \to C_* \cc G_{Nis}$ is a quasiisomorphism, where 

\[\tau_i(C_* \cc G)_{Nis} = \cdots \to C_{i-1} \cc G_{Nis} \to Z_i(C_* \cc G)_{Nis} \to 0 \to \cdots .\]

There is also an obvious morphism $f: \tau_i(C_* \cc G)_{Nis} \to H^i(C_* \cc G)_{Nis},$ which induces an isomorphism on the $i$th cohomology sheaf. It suffices to show that the morphism $f$ is equal to $0.$ Note that it defines a morphism in $D^-(Sm_{Nis})$ from $(C_* \cc G)_{Nis}$ to $H_i(C_* \cc G)_{Nis}.$ 

The presheaves $H^i(C_* \cc G)$ are homotopy invariant and quasistable presheaves with $\ZF_*$-trasnfers. By~\cite[Theorem 1.1]{GPHoInv},  $H^r(-,H^i(C_* \cc G)_{Nis})$ are homotopy invariant for all $r.$ 

Now, by~\cite[Proposition 12.19]{VoLec} for normal Nisnevich sheaves, 

\[\begin{tikzcd}
Hom_{D^-_{Nis}}((C_* \cc G)_{Nis},H^i(C_* \cc G)_{Nis} ) \arrow[d,equals] \\
 Hom_{D^-_{Nis}} (G_{Nis},H^i(C_* \cc G)_{Nis}) \arrow[d,equals] \\
Hom_{D^-_{Nis}} (0,H^i(C_* \cc G)_{Nis}) \arrow[d,equals] \\
~0,
\end{tikzcd}\]

since $\cc G_{Nis}=0.$ Therefore, $f=0,$ and $H^i(C_* \cc G)_{Nis}=0,$ proving the induction step.

\end{proof}

We also prove a homotopical version of the Mayer-Vietoris triangle:
\begin{proposition} \label{p:HoMV}
 
Let

\[\begin{tikzcd}
V-Z \arrow[d] \arrow[r,hook] & V \arrow[d]\\
X-Z \arrow[r,hook] & X\\
\end{tikzcd}\]
be a Nisnevich triangle in $Sm/k.$
\begin{enumerate}
\item{The natural morphisms

\[M_{fr} (V - Z) \to  M_{fr} (X - Z) \vee M_{fr} (V) \to M_{fr} (X) ,\]

induced by inclusions define a distinguished triangle in the stable homotopy category of simplicial Nisnevich sheaves.}
\item{Let $W \subset X$ be an open subscheme. The natural morphisms

\[\begin{tikzcd}
M_{fr}((V - Z,W \times_X V - Z \cap (W \times_X V))) \arrow[d] \\
M_{fr} ((X - Z,W-Z\cap W)) \vee M_{fr} ((V,W \times_X V)) \arrow[d] \\
 M_{fr} ((X,W)),
\end{tikzcd}\]

induced by inclusions of pairs define a distinguished triangle in the stable homotopy category of simplicial Nisnevich sheaves.}
\end{enumerate}
\end{proposition}

\begin{proof}
We give the proof for pairs, the scheme case is similar.

By Proposition~\ref{p:MV}, the morphism of preshaves of abelian groups

\begin{equation}\label{eq: C_*_MV}
\begin{tikzcd}
 \frac{C_* \ZF ((X - Z,W-Z\cap W)) \oplus C_* \ZF ((V,W \times_X V))}{ C_* \ZF((V - Z,W \times_X V - Z \cap (W \times_X V)))}  \arrow[d] \\
 C_* \ZF ((X,W)) 
\end{tikzcd}
\end{equation}

is a local quasiisomorphsim. 

Similarly to~\cite[Section 8]{GNP}, the $S^1$-spectra 

\begin{multline*}
LM_{fr}((X - Z,W-Z\cap W)),~ LM_{fr}((V,W \times_X V)), \\
LM_{fr}((V - Z,W \times_X V - Z \cap (W \times_X V))) \text{ and }  LM_{fr}((X,W)) 
\end{multline*}

 are Eilenberg-MacLane $S^1$-spectra of complexes 

\begin{multline*}
C_*\bb Z\F((X - Z,W-Z\cap W)),~ C_*\bb Z\F((V,W \times_X V)),\\
 C_*\bb Z\F((V - Z,W \times_X V - Z \cap (W \times_X V))) \text{ and } C_* \ZF ((X,W))
\end{multline*}

  correspondingly. Thus the morphism
\begin{equation*}
\frac{LM_{fr} ((X - Z,W-Z\cap W)) \oplus LM_{fr} ((V,W \times_X V))}{ LM_{fr} ((V - Z,W \times_X V - Z \cap (W \times_X V)))} \to LM_{fr} ((X,W)) 
\end{equation*}
induced by~\eqref{eq: C_*_MV}, is a local stable weak equivalence, 
hence such is the morphism

   \[\frac{\bb ZM_{fr} ((X - Z,W-Z\cap W)) \oplus \bb ZM_{fr} ((V,W \times_X V))}{ \bb ZM_{fr} ((V - Z,W \times_X V - Z \cap (W \times_X V)))} \to \bb ZM_{fr} ((X,W)),\]
	
by Theorem~\ref{ZM_fr_and_LM_fr}. The $S^1$-spectra 

\begin{multline*}
M_{fr} ((X - Z,W-Z\cap W)),~ M_{fr} ((V,W \times_X V)), \\
 M_{fr} ((V - Z,W \times_X V - Z \cap (W \times_X V))),~ M_{fr} ((X,W))
\end{multline*}

are $(-1)$-connected, as they are Segal spectra (see \cite[Definition 5.2]{GP}, \cite[Proposition 1.4]{Seg})

The stable Whitehead theorem~\cite[II.6.30(ii)]{Sch}
implies that the morphism
\begin{equation}\label{eq: M_fr-MV}
\frac{M_{fr} ((X - Z,W-Z\cap W)) \vee M_{fr} ((V,W \times_X V))}{M_{fr} ((V - Z,W \times_X V - Z \cap (W \times_X V)))} \to M_{fr}(M_{fr} ((X,W)))
\end{equation}
is a local stable weak equivalence.

\end{proof}

\end{document}